\documentclass[11pt]{article}

\usepackage{amsmath}
\usepackage{amsthm}
\usepackage{amsfonts}
\usepackage{algpseudocode}
\usepackage{algorithm}
\usepackage{bbm}
\usepackage{mathabx}
\usepackage{pythontex}
\usepackage{graphicx}
\usepackage{natbib}
\usepackage{dutchcal}
\usepackage{natbib}

\DeclareMathOperator*{\argmin}{arg\,min}
\DeclareMathOperator*{\tol}{tol}

\DeclareMathOperator*{\spn}{span}

\DeclareMathOperator*{\cl}{cl}

\DeclareMathOperator*{\sgn}{sgn}
\DeclareMathOperator*{\stp}{stop}
\DeclareMathOperator*{\href}{href}
\DeclareMathOperator*{\mg}{mg}
\newcommand*{\Bmin}{\texttt{Barrier.minimize}}

\newtheorem{theorem}{Theorem}[section]
\newtheorem{lemma}[theorem]{Lemma}
\newtheorem{corollary}[theorem]{Corollary}

\newtheorem{definition}[theorem]{Definition}

\newcommand\IG[2][]{\IfFileExists{#2}
  {\includegraphics[#1]{#2}}
  {\fbox{File #2 doesn't exist}%
   \message{Imagefile #2 doesn't exist^^J}}}

\newcommand{\cF}{\mathcal{F}}

\newcommand{\cf}{\mathcal{f}}

\newcommand{\cL}{\mathcal{L}}
\newcommand{\cQ}{\mathcal{Q}}

\newcommand{\cc}{\mathcal{c}}

\newcommand{\Cstop}{C_{\stp}}
\newcommand{\Chref}{C_{\href}}
\newcommand{\Cmg}{C_{\mg}}
\newcommand{\hH}{\hat{H}}

\algblockdefx[Parameters]{Parameters}{EndParameters}{\textbf{Parameters}}{}
\algblockdefx[Returns]{Returns}{EndReturns}{\textbf{Returns}}{}
\algblockdefx[Implementation]{Implementation}{EndImplementation}{\textbf{Implementation}}{}

\title{Algorithm MGB to solve highly nonlinear elliptic PDEs in $\tilde{O}(n)$ FLOPS.}
\author{S{\'e}bastien Loisel}

\begin{document}

\maketitle

\begin{abstract}
We introduce Algorithm MGB (Multi Grid Barrier) for solving highly nonlinear convex Euler-Lagrange equations. This class of problems includes many highly nonlinear partial differential equations, such as $p$-Laplacians. We prove that, if certain regularity hypotheses are satisfied, then our algorithm converges in $\tilde{O}(1)$ damped Newton iterations, or $\tilde{O}(n)$ FLOPS, where the tilde indicates that we neglect some polylogarithmic terms. This the first algorithm whose running time is proven optimal in the big-$\tilde{O}$ sense. Previous algorithms for the $p$-Laplacian required $\tilde{O}(\sqrt{n})$ damped Newton iterations or more.
\end{abstract}

\section{Introduction}

Let $\Omega \subset \mathbb{R}^d$ be an open domain. Let $\Lambda : \mathbb{R}^d \to \mathbb{R}$ be continuous and convex.
Let $W^{k,p}(\Omega)$ and $W^{k,p}_0(\Omega)$ denote the usual Sobolev spaces. Let $f(x)$ be an integrable forcing. Consider the Euler-Lagrange problem
\begin{align}
\inf_{u \in W^{1,\infty}_0(\Omega)} J(u) \text{ where } J(u)=\int_{\Omega} f(x)u(x) + \Lambda(\nabla u(x)) \, dx. \label{e:convex optimization}
\end{align}
Problem \eqref{e:convex optimization} is our model convex optimization problem in function space. We are using homogeneous Dirichlet conditions to streamline our presentation.

In convex optimization, for a given convex function $\Lambda(q)$, one defines the {\bf epigraph}
\begin{align}
Q & = \{ (q,s) \in \mathbb{R}^d \times \mathbb{R} \; : \; s \geq \Lambda(q) \}.
\end{align}
To the convex $d+1$-dimensional set $Q$, one may associate the infinite dimensional convex set $\cQ$ defined by
\begin{align}
\cQ = \{ (q(x),s(x)) \in L^{\infty}(\Omega;\mathbb{R}^d) \times L^{\infty}(\Omega) \; : \; s(x) \geq \Lambda(q(x))  \text{ a.e. } x \in \Omega \}. \nonumber
\end{align}

Denote by $D$ the differential operator (and $D^*$ its adjoint)
\begin{align}
D & = \begin{bmatrix}
\nabla \\ & 1
\end{bmatrix}
\text{ and }
D^* = \begin{bmatrix}
-\nabla \cdot \\
& 1
\end{bmatrix}.
\end{align}
Here, $\nabla$ denotes the gradient with respect to $x \in \Omega$, and $\nabla \cdot$ denotes the divergence.
Note that if $u(x) \in W_0^{1,\infty}(\Omega)$, then $\nabla u(x)$ is uniformly bounded on $\Omega$, and hence $\Lambda(\nabla u(x))$ is also uniformly bounded on $\Omega$ since $\Lambda$ is continuous.
Thus, denoting $z(x) = [u^T(x),s(x)]^T$, problem \eqref{e:convex optimization} is equivalent to
\begin{align}
\inf_{\substack{z \in W^{1,\infty}_0(\Omega) \times L^\infty(\Omega) \\ Dz \in \cQ}}
\int_{\Omega} c(x)[z(x)] \, dx \text{ where } c(x) = \begin{bmatrix} f(x) \\ 1 \end{bmatrix}. \label{e:convexopt}
\end{align}
Here and elsewhere, we use the notation $f[u]$, $f[u,v]$, $f[u,v,w]$, etc... for the application of a $k$-form $f$ to arguments $u,v,\ldots$ and we identify a vector $v$ with the corresponding form, so that $v[x] = v^Tx$. In a similar fashion, if $M$ is a matrix, then $M[u,v]$ is identified with $u^TMv$.

A barrier for $Q$ is a convex function $F(w)=F(q,s)$ that is finite on the interior $Q^{\circ}$ of $Q$ and such that $F(w) = +\infty$ for any $w \in \partial Q$. If $F(w)$ is a barrier for $Q$, then
\begin{align}
\cF(w) = \int_{\Omega} F(w(x)) \, dx,
\end{align}
is a barrier for $\cQ$.\footnote{Because function spaces have multiple inequivalent topologies, there are many different possible definitions of $\partial \cQ$. One such definition is to say that $z \in \partial \cQ$ if $z \in \cQ$ and $\cF(z) = \infty$. This definition of $\partial\cQ$ clearly depends on the choice of barrier.} To the problem \eqref{e:convexopt}, one associates the {\bf central path} $z^*(t,x)$, defined by:
\begin{align}
z^*(t,x) & = \argmin_{w(x) \in W^{1,\infty}_0(\Omega) \times L^{\infty}(\Omega)} \cf(t,w) \text{ where } \label{e:energy central path}\\
\cf(w,t) & = \int_{\Omega} tc(x)[w(x)] + F(Dw(x)) \, dx. \nonumber
\end{align}
We shall denote $z^*(t,x) = [u^*(t,x),s^*(t,x)]^T$.
Equation \eqref{e:energy central path} can be regarded as an ``energy minimization formulation'' for the central path.
The corresponding {\bf weak formulation}\footnote{Because our equations are highly nonlinear, it is not immediately obvious what is the most appropriate set of test functions $\{w(z)\}$. For simplicity, we use $w \in W^{1,\infty}_0 \times L^{\infty}$.} is obtained by a formal computation of the first variation of \eqref{e:energy central path}:
\begin{align}
\int_{\Omega} tc[w] + F'(Dz^*)[Dw] & = 0 \text{ for all } w \in W^{1,\infty}_0(\Omega) \times L^{\infty}(\Omega). \label{e:weak central path}
\end{align}
The {\bf strong formulation} is obtained by formal integration by parts:
\begin{align}
\begin{cases}
tc(x) + D^*(F'(Dz^*(t,x))) = 0 \text{ for } x \in \Omega \text{ and } \\ 
u^*(t,x) = g(x) \text{ for } x \in \partial \Omega. \label{e:strong central path}
\end{cases}
\end{align}
The PDE portion of \eqref{e:strong central path} is a nonlinear algebraic-elliptic system, as is revealed by the componentwise equations:
\begin{align}
tf(x) - \nabla \cdot F_u(\nabla u(t,x),s(t,x)) & = 0, \label{e:central path 1}\\
t+F_s(\nabla u(t,x),s(t,x)) & = 0. \label{e:central path 2}
\end{align}
To obtain a concrete solver for the problem \eqref{e:convex optimization}, we introduce a quasi-uniform triangulation $T_h$ of $\Omega$. We shall denote by $\int_{\Omega}^{(h)}$ a quadrature rule that approximates $\int_{\Omega}$ on the triangulation $T_h$.
We further introduce a piecewise polynomial finite element space $V_h$ on $T_h$ such that $DV_h$ has degree $\alpha-1$. Define
\begin{align}
\cf_h(w,t) & = \int_\Omega^{(h)} tc[w] + F(Dw) \text{ and}\\
z_h^*(t,x) & = \argmin_{z \in V_{h}} \cf_h(z,t). \label{e:z_h}
\end{align}

A basic algorithm for solving \eqref{e:convex optimization} is to approximately follow the central path $z_h^*(t,x)$ at discrete values of $t = t_k = \rho^kt_0$, where $\rho>1$ is the ``step size''. Concretely, given the (approximate) value of $z_h^*(t_{k-1},x)$, one finds $z_h^*(t_{k},x)$ by damped Newton iterations on the optimization problem \eqref{e:z_h}. This is called the ``return to the central path'' or simply ``return to the center''.

This strategy was first analyzed in \cite{loisel2020efficient} for the $p$-Laplacian, where it was shown that it converges in $\tilde{O}(\sqrt{n})$ damped Newton iterations; the tilde indicates that polylogarithmic terms are neglected. Here, $n = O(h^{-d})$ is the number of grid points. In that paper, it was also mentioned that if one uses an $\tilde{O}(n)$ FLOPS linear solver for the damped Newton steps, then one obtains an $\tilde{O}(n^{1.5})$ FLOPS solver for the $p$-Laplacian. In the present paper, we introduce a multigrid algorithm for solving the nonlinear problem \eqref{e:convex optimization} in $\tilde{O}(n)$ FLOPS, which is optimal in the $\tilde{O}$-sense, since it takes at least $\tilde{O}(n)$ operations simply to write out a solution to main memory.

Note that each damped Newton iteration requires the solution of a linear system whose structure is that of a moderately heterogeneous elliptic problem. In dimension $d=1$, this system is tridiagonal, so we may in fact perform each damped Newton step in $O(n)$ FLOPS. In dimension $d=2$, on the one hand, it is known that direct solvers cannot run faster than $O(n^{1.5})$ FLOPS \citep{hoffman1973complexity}. Despite this ``negative'' result, it is folkloric that, for the finite values of $n$ encountered on laptop computers, direct solvers effectively scale like $O(n)$ FLOPS, see e.g. \citet{380702}. Although direct solvers do not appear to achieve $O(n)$ performance in dimension $d\geq 2$, $H$-matrix-based solvers can indeed be used in all regimes and run in $\tilde{O}(n)$ FLOPS, see \citet{bebendorf2016low} and references therein. We also mention domain decomposition preconditioners, e.g. \citet{loisel2015optimized}, \citet{loisel2013condition}, \citet{subber2014schwarz}, \citet{loisel2010optimized}, \citet{karangelis2015condition}, \citet{greer2015optimised}, \citet{loisel2017optimal}, \citet{loisel2008hybrid}. For nonlinear problems, see also \citet{berninger20142}. Such methods are also related to subspace correction methods \citep{tai2002global}.

In \cite{loisel2020efficient}, the theoretically optimal step size is found to be the well-known ``short step'' of convex optimization, which here is $\rho-1 \sim n^{-0.5} \sim h^{d/2}$. In convex optimization, one often prefers to use a step size $\rho$ that is independent of $n$, this is called a ``long step'', for example $\rho = 2$. If, by luck, each return to the center requires $O(1)$ damped Newton iterations, then the long step approach is a clear winner. Unfortunately, the theory of convex optimization predicts that long-stepping schemed may need as many as $O(n)$ damped Newton steps to return to the center. In \citet{loisel2020efficient}, it was revealed that the $p$-Laplacian triggers this worst-case behavior in long-stepping schemes. In order to obtain the ``best of both worlds'', an adaptive algorithm was devised in that paper, but the theoretical performance estimate is still governed by the short-step scheme.

In the present paper, we introduce Algorithm MGB, which converges to the solution in $\tilde{O}(1)$ Newton steps, and $\tilde{O}(n)$ FLOPS. This is clearly optimal in the big-$\tilde{O}$ sense since it takes time $\tilde{O}(n)$ to simply write out the solution to main memory. Algorithm MGB achieves this performance by returning to the center in $\tilde{O}(1)$ damped Newton steps, and the $t$ step size is long, in the sense that $(\rho-1)^{-1} = \tilde{O}(1)$. Although the theory says that the $t$ step size may depend polylogarithmically on $n$, we found in our numerical experiments that the $t$ step sizes were completely independent of $n$.

Iterative numerical algorithms are often parametrized in terms of the problem size $n$ and the tolerance $\tol>0$ used to stop the iteration. However, in our case, the number $\tol$ can be expressed in terms of $n$, as we now describe.

Although the function $J(u)$ in \eqref{e:convex optimization} is unlikely to be globally twice differentiable, it often happens that $J''(u^*)[v^2]$ is indeed well defined near a minimizer $u^*$ of \eqref{e:convex optimization}, provided the test function $v$ is sufficiently regular. When this happens, then $J(v_h)-J(u^*) \sim h^{2\alpha}$ if $v_h-u^* \sim h^\alpha$. If $v_h = u_h^*$ is obtained by a central path, i.e. $z_h^*(t,x) = [u_h^*(t,x),s_h^*(t,x)]$, then we have 
\begin{align}
J(u_h^*(t,x))-J(u^*(x))
& =
J(u_h^*(t,x)) - J(u_h^*(\infty,x)) + J(u_h^*(\infty,x))-J(u^*(x)) \nonumber \\
&  = O(t^{-1}+h^{2\alpha}), \label{e:t and h}
\end{align}
see Lemma \ref{l:filter}.
Equilibrating the error terms produces a termination criterion for the central path: $t^{-1} \sim h^{2\alpha}$. 

In the multigrid method, we have a sequence of grid parameters $1 \geq h^{(1)} > \ldots > h^{(L)} > 0$ and corresponding quasi-uniform triangulations $T_{h^{(\ell)}}$. To simplify our exposition, the grid $T_{h^{(\ell+1)}}$ is obtained by bisecting the edges of $T_{h^{(\ell)}}$ so that $h^{(\ell+1)} = 0.5 h^{(\ell)}$ and the grids are automatically quasi-uniform. To each grid $T_h$, we associate the finite element space $V_{h} \subset W^{1,\infty}(\Omega) \times L^{\infty}(\Omega)$, and we have that $V_{h^{(1)}} \subset V_{h^{(2)}} \subset \ldots \subset V_{h^{(L)}}$.

\begin{figure}
\includegraphics[width=0.5\textwidth]{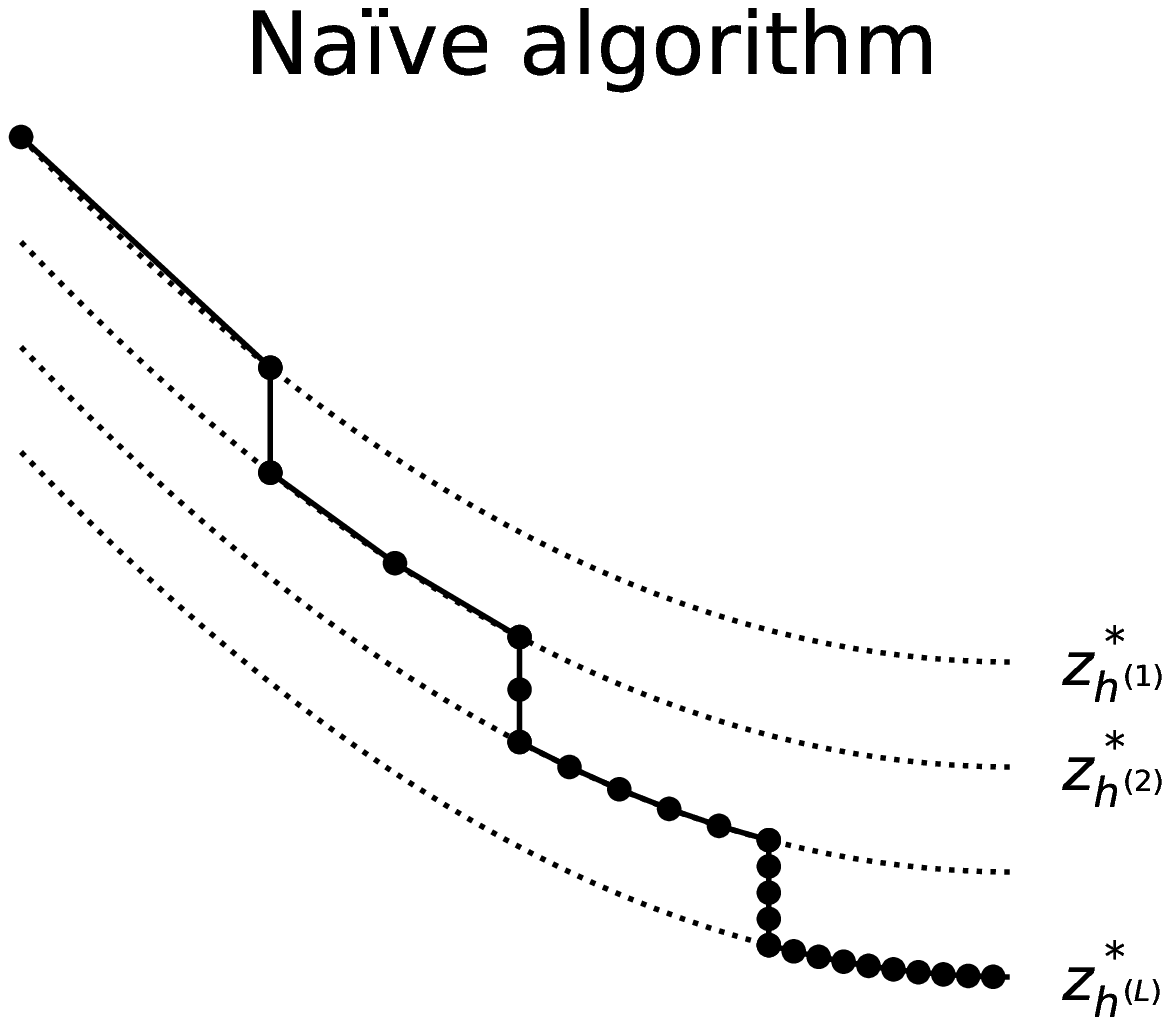}\includegraphics[width=0.5\textwidth]{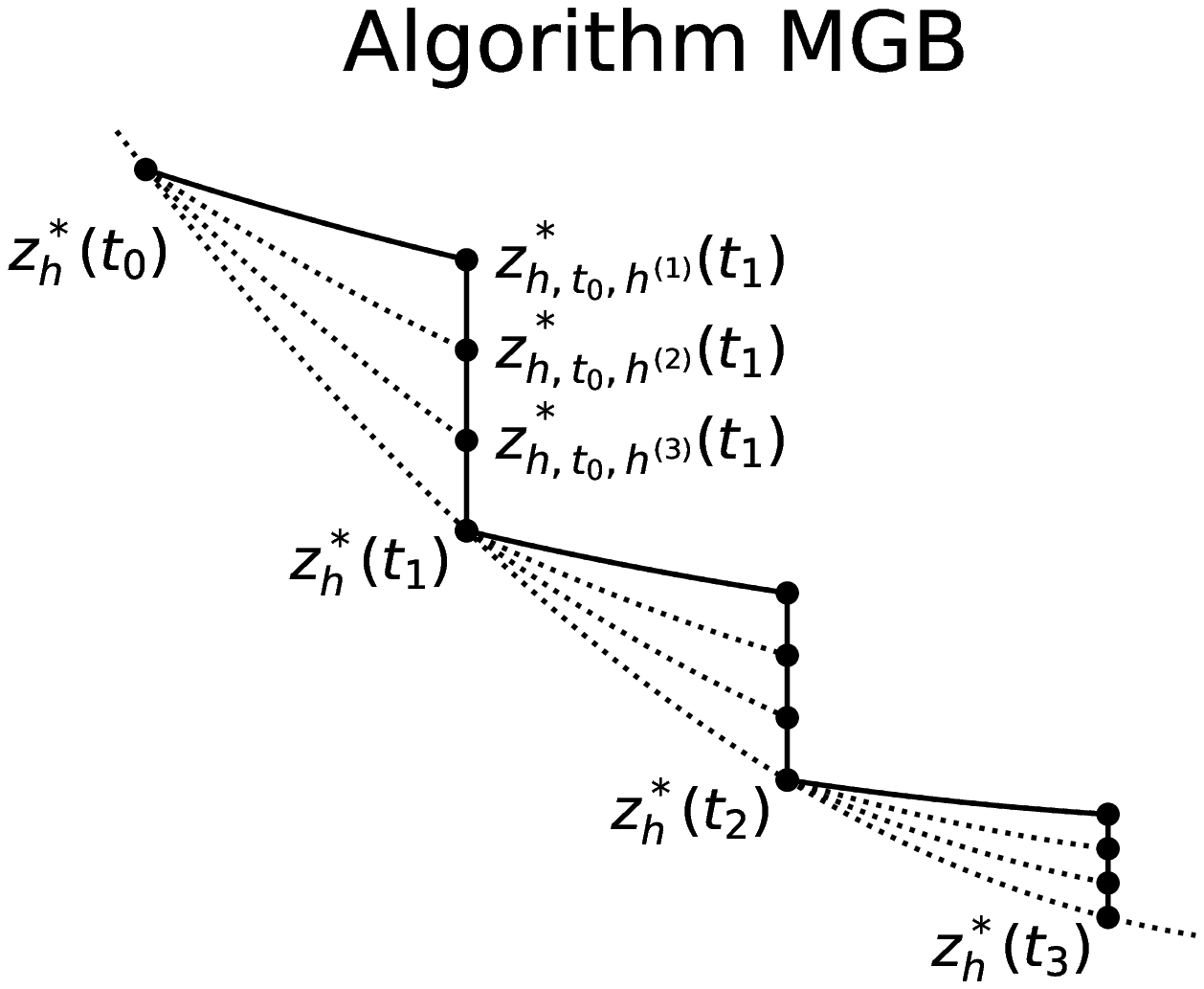}
\caption{
The naïve algorithm (left) shrinks the grid parameter $h$ ``as needed'' as $t$ increases, converging in $\tilde{O}(\sqrt{n})$ damped Newton steps (marked by circles), mostly on the fine grid. Our Algorithm MGB (right) requires $\tilde{O}(1)$ damped Newton steps. \label{f:algorithms}
}
\end{figure}

\begin{definition}[Naïve multigrid algorithm]
Let $t_0 = O(1)$ and $h_0=h^{(1)}$ be the coarsest grid parameter. Assume $V_h$ is a piecewise polynomial finite element space, and the degree of $DV_h$ is $\alpha-1$. Assume that $\alpha \geq d$. Assume $z^{(0)} \in g+V_{h_0}$. For $j=1,2,\ldots$, given $(t_{j-1},h_{j-1})$, define inductively $(t_k,h_k)$ using, at each step, either $t$-refinement or $h$-refinement, as follows.
\begin{itemize}
\item {\bf $t$ refinement:} set $(t_{j},h_{j}) = (\rho_j t_{j-1},h_{j-1})$, where $\rho_j >1$ is a $t$-step size.
\item {\bf $h$ refinement:} set $(t_j,h_j) = (t_{j-1},{1 \over 2}h_{j-1})$, where the factor $1/2$ is imposed by the grid subdivision scheme.
\end{itemize}
Then, $z^{(j)}$ is found by damped Newton iterations so as to minimize $\cf_h(z,t_j)$ over $z \in g+V_{h_j}$ with initial guess $z = z^{(j-1)}$. The stopping criterion is $t_j > \Cstop h^{-2k}$.

The {\bf schedule} is the choice of the order of $h$ and $t$ refinements, i.e. the sequence $(t_j,h_j)$. The ``$h$-then-$t$'' schedule is $t_0=\ldots=t_{L-1} = t_0$ and $\{h_j\}_{j=0}^{L-1} = \{h^{(1)},\ldots,h^{(L)}\}$ (i.e. $h$ refinements), followed by $h_L = h_{L+1} = \ldots = h^{(L)}$ and $t_j = \rho_j t_{j-1}$ for $j=L,L+1,\ldots$ (i.e. $t$ refinements). In other words, the $h$-then-$t$ schedule performs all $h$ refinements first, and then all $t$ refinements are done on the finest grid level $h = h^{(L)}$.
\end{definition}

The ``naïve multigrid algorithm'' is depicted in Figure \ref{f:algorithms}, left.
This method is similar to the method described by \citet{schiela2011interior}, where the authors write that ``This allows to perform most of the required Newton steps on coarse grids, such that the overall computational time is dominated by the last few steps.'' Unfortunately, we have found this not to be the case, as we now explain with our first main theorem.

\begin{theorem}[Newton iterations of naïve multigrid] \label{t:naive iteration count}
Assume $(T_h,c,F)$ is regular, as per Definition \ref{d:regular}. Let $t_0 = O(1)$. The ``$h$-then-$t$'' schedule converges in
\begin{align}
\tilde{O}(h^{-0.5d}) = \tilde{O}(n^{0.5}) \text{ damped Newton iterations,} \label{e:naive iteration count}
\end{align}
where $n$ is the number of grid points in $T_h$.

For an arbitrary schedule of $h$ and $t$ refinements, we estimate the number of damped Newton iterations on the finest grid. The best possible theoretical estimate is $\tilde{O}(n^{0.5})$ iterations. In other words, the $h$-then-$t$ schedule is optimal in the big-$\tilde{O}$ sense.
\end{theorem}
We delay the proof of our main theorems to later sections.

We also mention that there is seemingly no satisfactory analysis of multigrid algorithms for nonlinear problems in the literature. Indeed, \citet{brabazon2014nonlinear} write that ``To the best of our knowledge there exists no valid theory for the convergence of FAS for the case where the nonlinearity is in the highest order term'' (see also references therein). This is despite the methods in \cite[Chapter 13]{hackbusch2013multi}.

In order to describe our algorithm, we introduce the notion of {\bf shifted central paths}. For given $H \geq h$ and $t_0>0$, define
\begin{align}
z^*_{h,t_0,H}(t) & = \argmin_{z \in z^*_{h}(t_0) + V_H} \cf_h(z,t). \label{e:shifted central path}
\end{align}
As before, the parameter $x \in \Omega$ omitted in the notation is implied.

\begin{definition}[Algorithm MGB] \label{d:Algorithm MGB}
Let $t_0 = O(1)$, $z^{(0)} \in V_{h^{(1)}} \cap \cQ$ be given.
Let $h = h^{(L)}$ be the fine grid. We define $z^{(k)}$ inductively as follows.

Given an approximation $z^{(k)} \approx z^*_h(t_k)$, compute $z^{(k+1)} \approx z^*_h(t_{k+1})$ as follows, where $t_{k+1} = \rho_k t_k$ and $\rho_k > 1$ is some step size.

\begin{itemize}
\item For $\ell = 1,\ldots,L$, find an approximation $z^{(k+{\ell \over L})}$ of  $z^*_{h,t_k,h^{(\ell)}}(t_{k+1})$ by damped Newton iterations on \eqref{e:shifted central path}, starting with the initial guess $z = z^{(k+{\ell-1 \over L})}$.
\end{itemize}

Stop when $t_j > \Cstop h^{-2\alpha}$.
\end{definition}

Our second main result shows that Algorithm MGB converges in $\tilde{O}(1)$ Newton iterations.

\begin{theorem}[Newton iterations of Algorithm MGB] \label{t:Algorithm MGB}
Assume $(T_h,c,F)$ is regular, as per Definition \ref{d:regular}. There is a long step size $\rho_k = t_{k+1}/t_k > 1$ (i.e. $(\rho-1)^{-1} = \tilde{O}(1)$) such that Algorithm MGB converges in
\begin{align}
\tilde{O}(1) \text{ damped Newton iterations.}
\end{align}
\end{theorem}

Using an $\tilde{O}(n)$ FLOPS linear solver for the damped Newton steps, our algorithm thus requires $\tilde{O}(n)$ FLOPS to solve \eqref{e:convex optimization}, which is optimal in the big-$\tilde{O}$ sense.

We also mention that in the convex optimization literature, long-stepping schemes generally require $O(n)$ iterations per centering step, which is often thought to be optimal for long-stepping schemes for general convex optimization problems. Within that framework, our result can be interpreted as a description of a wide class of convex optimization problems, for which return to the center in a long-stepping scheme can be achieved in $\tilde{O}(1)$ damped Newton steps. We are not aware of any other long-stepping scheme that has this extremely fast convergence property.

Our regularity hypotheses are described in Definition \ref{d:regular}. It says in part that $Dz_0^*(\tau^{-1},x)$ should be sufficiently smooth as a function of $\tau = t^{-1}$ and $x \in \Omega$. Such smoothness hypotheses are commonplace in the analysis of PDEs. Mathematicians are usually comfortable with regularity hypotheses of a well-known type; solutions of PDEs are expected to be smooth. However, one of the requirements of Definifion \ref{d:regular} is that $\sqrt{F''(Dz_{h,t_0,H}^*(t))}$ should be in some uniform reverse Hölder class. 

We are not aware of any other algorithm that uses reverse Hölder classes as a regularity hypothesis. In order for such a class to be suitable as a regularity hypothesis, these functions should somehow be plentiful, especially as solutions of PDEs.  The reverse Hölder classes are studied in the context of classical Fourier analysis \citep{grafakos2008classical}, \citep{cruz1995structure} and have applications to the analysis of nonlinear PDEs \citep{kinnunen2000higher}.

In this paper, the symbol $C$ shall be used to denote a generic constant that may depend on $\Omega$, $c$, $F$ or the regularity parameter of $T_h$ or the degree $\alpha-1$ of the piecewise polynomial functions $DV_h$, but is otherwise independent of $h$ and $t$. The constant $C$ may not represent the same number from one equation to the next. If multiple constants are involved, we may use the notation $C_1,C_2,\ldots$ to distinguish them.

Our paper is organized as follows. In Section 2, we briefly review the theory of self-concordant calculus in finite dimensions, and collect some essential results for the special case of the epigraph $Q$. In Section 3, we introduce the quadrature rules that we will be using. In Section 4, we give a abstract theory of self-concordance in function spaces, culminating with the notion of ``reverse Hölder continuation''. In Section 5, we use this theory to analyze the behavior of the naïve algorithm, while in Section 6, we analyze Algorithm MGB. In section 7, we give an ``a priori estimate'' for the reverse Hölder inequality that is part of our regularity hypotheses. This shows that the reverse Hölder inequality can be obtained from some smoothness conditions on the solution and on the barrier. In Section 8, we discuss the practical MGB algorithm, which introduces some optimizations and methods for handling floating point inaccuracies. We have some numerical experiments in Section 9, and we end with some conclusions in Section 10.

\section{Self-concordance in finite dimensions}

If $\|\cdot\|_+$ is any norm on $\mathbb{R}^n$, then for each $k = 1,2,3,\ldots$, there is an induced norm on homogeneous polynomials of degree $k$, defined by $\vvvert P \vvvert_+ = \sup |P[u^k]|/\|u\|_+^k$, where the supremum is taken over $u \neq 0$. If $\|\cdot\|_+$ is a Hilbertian norm, then \cite{banach1938homogene} showed that $\vvvert \cdot \vvvert_+$ is also given by $\vvvert P \vvvert_+ = \sup |P[u^{(1)},\ldots,u^{(k)}]|/\prod_j \|u^{(j)}\|_+$, with the supremum taken over nonzero vectors,  see also \cite{loisel2001polarization}.

Let $F(z)$ be a strictly convex thrice-differentiable barrier function on $Q$. In particular, $H(z) = F''(z)$ is symmetric positive definite (SPD) for every $z \in Q^{\circ}$, and thus $(Q^{\circ},H)$ is a Riemannian manifold. The norm at $z \in Q^{\circ}$ is denoted $\|u\|_z = \|u\|_{F''(z)} = \sqrt{u^TF''(z)u}$. Recall that $F$ is said to be {\bf standard self-concordant} if
\begin{align}
|F'''(z)[u^3]| \leq 2\|u\|_z^3. \label{e:standard self-concordant}
\end{align}
Equation \eqref{e:standard self-concordant} says that the induced norm $\vvvert F'''(z)\vvvert_{z}$  is at most $2$.

A standard self-concordant function is further said to be a {\bf self-concordant barrier} with parameter $1 \leq \nu < \infty$ if
\begin{align}
|F'(z)[u]| \leq \sqrt{\nu} \|u\|_z. \label{e:self-concordant barrier}
\end{align}
Equation \eqref{e:self-concordant barrier} says that the induced norm $\vvvert F'(z) \vvvert_z$ is at most $\sqrt{\nu}$.
The dual norm is $\|u\|_z^* = \sqrt{u^T[F''(z)]^{-1}u}$ and satisfies the ``Cauchy-Schwarz'' type inequality $|u^Tv| \leq \|u\|_x^* \|v\|_x$.
The {\bf Newton decrement} is $\lambda_F(z) = \|F'(z)\|_x^* = \sqrt{F'(z)^T[F''(z)]^{-1}F'(z)}$ and one checks that \eqref{e:self-concordant barrier} is equivalent to $\lambda_F^2(z) \leq \nu$.

In order to avoid repeating the well-known theory, we shall frequently quote results from the book of \cite{nesterov1994interior}, e.g. [NN, Theorem 2.5.1] states that there exists a $\nu$-self-concordant barrier for the convex set $Q$, with $\nu = O(d)$, and the equation (NN2.5.1) is a formula for the ``universal barrier''.

On the value of $\nu$ in dimension $n$, [NN, Section 2.3.4] states that ``generally speaking, the parameter cannot be less than $n$''. As a result, even if $F(z)$ is a self-concordant barrier, still $\cF(w)$ will not be a self-concordant barrier, since its domain $\cQ$ is infinite dimensional, and hence the parameter of $\cF$ would have to be $\infty$.

Even though $\cF$ is not a self-concordant barrier for $\cQ$, still a variant of \cite[Theorem 4.2.7]{nesterov2013introductory} holds for our central path $z^*(t,x)$, provided $F(z)$ is a $\nu$-self-concordant barrier on $Q$. As we shall see in Lemma \ref{l:filter}, our central path $z^*(t,x)$ satisfies
\begin{align}
\cc[z^*] - \cc^* \leq {|\Omega| \nu \over t} \text{ where } \cc[w] = \int_{\Omega} c(x)[w(x)] \, dx. \label{e:filter}
\end{align}
As usual, $|\Omega|$ is the Lebesgue measure of $\Omega$, and we have denoted by $\cc^*$ the infimum in \eqref{e:convexopt}. The bound \eqref{e:filter} shows that $z^*(t,\cdot)$ is a minimizing filter in $\cQ$ for the functional $\cc$. In other words, the central path can also be used in infinite dimensions to solve convex optimization problems.

Self-concordant functions satisfy the following explicit bounds.
\begin{lemma} \label{l:basic self concordant estimates}
Let $G(z)$ be a standard self-concordant function on $Q$.
Define
\begin{align}
\psi(\alpha) & = \alpha-\log(1+\alpha).
\end{align}
Denote $\|u\|_v^2 = G''(v)[u^2]$.
For $y,z \in Q$,
\begin{align}
{
\|y-z\|_z^2
\over
(1+\|y-z\|_z)^2
} &
\leq \|y-z\|_y^2 
\text{ and } \\
G(z) + G'(z)[y-z] + \psi(\|y-z\|_z) & \leq G(y). \label{e:sc norm ineq 1}
\end{align}
{Furthermore, if $\|y-z\|_z < 1$, then}
\begin{align}
\|y-z\|_y^2 
& \leq {
\|y-z\|_z^2
\over
(1-\|y-z\|_z)^2
}
\text{ and } \label{e:sc norm ineq 2} \\
G(y) & \leq G(z) + G'(z)[y-z] + \psi(-\|y-z\|_z) \label{e:sc norm ineq 3}
\end{align}
\end{lemma}
\begin{proof}
The following argument is standard.
Let $\phi(t) = G(z+th)$ with $h = y-z$. Thus, $\phi''(t) = G''(z+th)[h^2]$ and
\begin{align}
|\phi'''(t)| & = |G'''(z+th)[h^3]| \\
& \leq 2(G''(z+th)[h^2])^{1.5} = 2 \phi''(t).
\end{align}
Solving the extremal differential equations $\phi'''(t) = \pm 2\phi''(t)^{1.5}$ for the unknown $\phi''(t)$ yields
$
-t  \leq \phi''(t)^{-0.5} - \phi''(0)^{-0.5} \leq t,
$ and hence
\begin{align}
{\phi''(0) \over (1+t\phi''(0)^{0.5})^2}
\leq
\phi''(t)
\leq
{\phi''(0) \over (1-t\phi''(0)^{0.5})^2}.
\end{align}
The lower bound is valid for $t \geq 0$ and the upper bound is valid for all $0 \leq t < \phi''(0)^{-0.5}$. Integrating twice, we find that
\begin{align}
\psi(t\phi''(0)) & \leq \phi(t)-\phi(0)-\phi'(0)t \leq \psi(-t\phi''(0)).
\end{align}
The results follow by substituting $t=1$.
\end{proof}

We now collect several basic properties from the literature.

\begin{lemma} \label{l:properties}
Assume that $\Lambda(q)$ is convex and satisfies $\Lambda(q) \geq \alpha \|q\|_2+\beta$ for some $\alpha>0$, $\beta \in \mathbb{R}$ and all $q \in \mathbb{R}^d$; in other words, the graph of $\Lambda$ lies above some cone.
Denote $z = (q,s) \in Q$. For any $q$ and $t>0$ let $s^{(t)}(q)$ satisfy $F_s(q,s^{(t)}(q))+t = 0$. Then, $s^{(t)}(q)$ is uniquely defined for any $q$ and any $t>0$. Furthermore, let $K \subset \mathbb{R}^d$ be compact. Let $t_0>0$. There is a constant $C = C(K,F,t_0)$, $1 \leq C < \infty$, such that the following are true:
\begin{enumerate}
\item $F_s(q,s)$ is a monotonically increasing function of $s$.
\item For $\epsilon>0$, 
$
-{\nu \over \epsilon} \leq
F_s(q,\Lambda(q)+\epsilon) \leq -{1 \over \epsilon}.
$
\item ${1 \over t} \leq s^{(t)}(q)-\Lambda(q) \leq {\nu \over t}$.
\item Let $q \in K$ and $\epsilon := s-\Lambda(q)$. If $0 < \epsilon \leq {\nu \over t_0}$ then $\sigma(F''(q,s)) \subset [C_1,C_2\epsilon^{-2}]$, where $\sigma(\cdot)$ denotes the spectrum of a matrix.
\item $F(z) = -\log \Phi(z)$ with $-\Phi^{1 \over \nu}(z)$ convex and $\{z \; : \; \Phi(z)>0 \}$ is the interior of $Q$. 
\item For $x \in Q$ and $y \in \mathbb{R}^{d+1}$, put $r = \|x-y\|_x$. If $r<1$ then $y \in Q$ and also
\begin{align}
(1-r)^2 F''(x) \preceq F''(y) \preceq {1 \over (1-r)^2} F''(x). \label{e:Hessians with r}
\end{align}
\end{enumerate}
\end{lemma}
\begin{proof}
Recall the function $\pi_y(x)$ [NN] defined by
\begin{align}
\pi_y(x) = \inf \{ t \geq 0 \; : \; \overbrace{y+t^{-1}(x-y)}^w \in Q \}.
\end{align}
The fact that $s^{(t)}(q)$ is uniquely defined for any $t>0$ follows from properties 1--3. We now show all the numbered properties in order.
\begin{enumerate}
\item $F_s$ is a monotonically increasing function of $s$ because $F$ is strictly convex. 
\item Put $z = (r,\Lambda(r))$ and $x = (r,\Lambda(r)+\epsilon)$ (note that  $\pi_z(x) = 0$) into (NN2.3.7) to find $F_s(r,\Lambda(r)+\epsilon)(-\epsilon) \geq 1$. Conversely, put $y = \Lambda(r)^+$ and $x = \Lambda(r)+\epsilon$ into (NN2.3.2) to find $F_s(r,\Lambda(r)+\epsilon)(-\epsilon) \leq \nu$.
\item From the first two properties, $s^{(t)}(r)$ is uniquely defined. From property 2, $t+F_s(r,\Lambda(r)+\epsilon) \leq t-{1 \over \epsilon} < 0$ if $\epsilon<1/t$ thus $s^{(t)}(r) \geq \Lambda(r) + {1 \over t}$. Also from property 4, $t+F_s(r,\Lambda(r)+\epsilon) \geq t-{\nu \over \epsilon} > 0$ if $\epsilon > {\nu \over t}$.  Thus, $s^{(t)}(r) \leq \Lambda(r) + {\nu \over t}$.
\item Now we use (NN2.3.9):
\begin{align}
F''(z)[h^2]^{-0.5} & \leq q_z(h) \leq (1+3\nu)F''(z)[h^2]^{-0.5}, \label{e:NN2.3.9}
\end{align}
where $q_z(h) = \sup \{ t \; : \; z \pm th \in Q\}$. The function $q_z(h)$ measures the distance from point $z$ to some point $z\pm th$ on $\partial Q$ in the directions $\pm h$.

Since the graph of $\Lambda$ lies above some cone, i.e. $\Lambda(q) \geq \alpha\|q\|_2+\beta$, we can find an upper bound for $q_z(h)$. To do this, write $z = (q,s)$ and $h = (w,\eta)$. From $z \pm th \in Q$, we find that $s \pm t\eta \geq \Lambda(q \pm tw) \geq \alpha \|q \pm tw\|_2 + \beta \geq \alpha(t \|w\|_2-\|q\|_2)+\beta$. Picking the sign so that $\pm t\eta<0$ yields
\begin{align}
t \leq {s+\alpha\|q\|_2-\beta \over (\alpha \|w\|_2 + |\eta|)}
\leq {s+\alpha\|q\|_2-\beta \over C\|h\|_2}, \label{e:t upper bound}
\end{align}
where we have used that $\alpha\|w\|_2+|\eta|$ is a norm of $h = (w,\eta)$ that is equivalent to the Euclidian norm $\|h\|_2$ by norm equivalence in finite dimensions. Furthermore, the numerator of \eqref{e:t upper bound} is also bounded, since $z = (q,s)$ ranges in some compact set. This shows that $t$ is uniformly bounded by $C/\|h\|_2$, and hence
\begin{align}
q_z(h) & \leq {C \over \|h\|_2}. \label{e:q_z upper bound}
\end{align}

We now find a lower bound for $q_z(h)$. For $s > \Lambda(q)$, put $z = (q,s)$. Since $\Lambda$ is Lipschitz over $K$, the epigraph of $\Lambda$ contains a ball of radius $C(s-\Lambda(q)) = C\epsilon$ centered at $z$. Thus, 
\begin{align}
q_z(h) & \geq C\epsilon/\|h\|_2. \label{e:q_z lower bound}
\end{align}
The result follows from \eqref{e:NN2.3.9}, \eqref{e:q_z upper bound}, \eqref{e:q_z lower bound}.
\item This is part (iv) of [NN, Proposition 2.3.2]. 
\item This is [NN, Theorem 2.1.1].
\end{enumerate}
\end{proof}

\begin{corollary} \label{c:properties}
\begin{align}
\sigma(F''(Dz^*_0(t,x))) \subset [C_1,C_2 t^{2}], \label{e:spectrum of F''}
\end{align}
where $\sigma(M)$ denotes the spectrum of the matrix $M$.
\end{corollary}


For self-concordant functions, damped Newton iterations can be proven to converge, as follows.

\begin{lemma} \label{l:boyd}
Let $G(z)$ be convex with a minimum at $z^*$.
Define $\cL(\delta) = \{ w \; : \; G(w) - G(z^*) \leq \delta \}$.
Assume that $G(z)$ is standard self-concordant on $\cL(\delta)$.
Let $z^{(0)} \in \cL(\delta)$ be given, we define the {\bf suboptimality gap} as $G(z^{(0)}) - G(z^*) \leq \delta$. Define the sequence $z^{(k)}$ for $k \geq 1$ by damped Newton iterations. Assume we stop the damped Newton iterations at the first iteration $k$ such that $G(z^{(k)}) - G(z^*) < \epsilon$. Then,
\begin{align}
k & \leq C(G(z^{(0)}) - G(z^*)) + \log_2 \log_2 \epsilon^{-1}.
\end{align}
\end{lemma}
\begin{proof}
This is \cite[(9.56)]{boyd2004convex}, which also reveals that $C \leq 375$. A much smaller value of $C$ can be estimated by following [NN, section 2.2.3]. In double precision arithmetic, the expression $\log_2 \log_2 \epsilon^{-1}$ can be bounded by $6$.
\end{proof}

\section{Quadrature}

\begin{definition}
Let $\hat{K} \subset \mathbb{R}^d$ be the ``reference simplex''. For each $h>0$, assume $T_h$ is a triangulation of $\Omega$. For each $K \in T_h$, we associate an affine map $f_K(x) = A_Kx+b_K$ such that $K = f_K(\hat{K})$. We say that $T_h$ is quasi-uniform with parameter $\rho \geq 0$ if $\vvvert A_K \vvvert \leq h$ and $\vvvert A_K^{-1}\vvvert^{-1} \leq \rho h$.
\end{definition}

Assume we have a quadrature rule for the reference simplex:
\begin{align}
I_{\hat{K}} \eta(x) & := \sum_{j=1}^{\beta} \omega_j \eta(x_{\hat{K},j}).
\end{align}
We shall require that the quadrature weights satisfy
\begin{align}
\omega_j > 0. \label{e:positive weights}
\end{align}
From this, we obtain a quadrature rule for $K \in T_h$:
\begin{align}
\int_K^{(h)} \eta(x) & = |\det A_K| I_{\hat{K}} \eta(A_K x + b_K).
\end{align}
Then, if $E$ is a union of simplices in $T_h$,
\begin{align}
\int_E^{(h)} \eta(x) & = \sum_{\substack{K \in T_h \\ K \subset E}} \int_K^{(h)} \eta|_K.
\end{align}
Here, the notation $\eta|_K$ indicates that we first restrict $\eta$ to $K$, and it suffices for $\eta$ to be continuous on each $K \in T_h$. This allows one to integrate a function $\eta$ which may have jump discontinuities on edges of $T_h$. In particular, if $E$ is the union of simplices in $T_h$, then
\begin{align}
\int_E^{(h)} 1 & = |E|.
\end{align}
We shall also denote the exact integral by $\int_E^{(0)} = \int_E$.
Because of \eqref{e:positive weights}, we may write the quadrature rule as
\begin{align}
\int_\Omega^{(h)} \eta & = \sum_{\substack{K \in T_h \\ k=1,\ldots,\beta}} \omega_{K,k} y_{K,k} \text{ where}\\
\omega_{K,k} & = |\det A_K|\omega_k, \; x_{K,k} = A_K x^{(k)}+b_K \text{ and } y_{K,k} = \eta|_K(x_{K,k}).
\end{align}
Thus,
\begin{align}
C_{\min} h^d \leq \{\omega_{K,k},|K|\} \leq C_{\max}h^d, \label{e:quadrature weights}
\end{align}
for some constants $0<C_{\min}<1<C_{\max}<\infty$.
We thus have Jensen's inequality
\begin{align}
\psi\left(
{1 \over |E|} \int_{E}^{(h)} \eta
\right) &
\leq
{1 \over |E|}\int_{E}^{(h)} \psi(\eta),
\end{align}
for any convex function $\psi$. We further define the discrete norms
\begin{align}
\|\eta\|_{L_h^p(E)}^p & = \int_E^{(h)} |\eta|^p \text{ for } 1 \leq p < \infty, \\
\|\eta\|_{L_h^\infty(E)} & = \sup_{\substack{K \in T_h \\ K \subset E}} |\eta(x_{K,k})|.
\end{align}
We then have the discrete Hölder inequalities:
\begin{align}
\int_E^{(h)} |\eta| \leq \|\eta\|_{L^p_h(E)} \|\eta\|_{L^{p'}_h(E)} \text{ where } {1 \over p} + {1 \over p'} = 1.
\end{align}
The discretizations of $\cF$ and $\cf$ are:
\begin{align}
\cF_h(w) & = \int_{\Omega}^{(h)} F(w(x)) \text{ and } \\
\cf_h(w) & = \int_{\Omega}^{(h)} tc(x)[w(x)] + F(Dw(x)).
\end{align}
Let $V_h \subset W^{1,\infty}(\Omega) \times L^{\infty}(\Omega)$ be a piecewise polynomial space, such that $DV_h$ is of degree $\alpha-1$. We define
\begin{align}
\cQ_h = \{ v \in V_h \; : \; v(x_{K,k}) \in Q \text{ for all } K \in T_h, \; k=1,\ldots,\beta \}.
\end{align}
Then, we define the discretizations
\begin{align}
z_h^*(t) & = \argmin_{z \in V_h} \cf_h(z,t) \text{ and} \\
z_{h,t_1,H}^*(t) & = \argmin_{z \in z_h^*(t_1) + V_H} \cf_h(z,t).
\end{align}
Note that
\begin{align}
\int_\Omega^{(h)} tc[\phi_h] + F'(Dz_h^*(t))[D\phi_h] & = 0 \text{ for all } \phi_h \in V_h \text{ and} \\
\int_\Omega^{(h)} tc[\phi_H] + F'(Dz_{h,t_1,H}^*(t))[D\phi_H] & = 0 \text{ for all } \phi_H \in V_H. \label{e:discrete weak form 2}
\end{align}
Note that \eqref{e:discrete weak form 2} uses the quadrature $\int_\Omega^{(h)}$ on the fine grid $T_h$, but the test function $\phi_H \in V_H$ is on the coarse grid $T_H$, $H \geq h$.

\section{Self-concordance in function spaces}

We were not able to find the system \eqref{e:central path 1}, \eqref{e:central path 2} in the literature on nonlinear elliptic partial differential equations, but see \cite{gilbarg1977elliptic}. In the present paper, we shall assume that the central path $z^*(t,x) = z^*_0(t,x)$ exists and is unique for $t>0$ and solves \eqref{e:energy central path}, \eqref{e:weak central path}, \eqref{e:strong central path}, \eqref{e:central path 1}, \eqref{e:central path 2}.

\begin{lemma} \label{l:filter}
Let $h \geq 0$.
Assume $F(z)$ is a self-concordant barrier for $Q$ with parameter $\nu$. Then,
\begin{align}
\int_\Omega^{(h)} c[z_{h,t_0,H}^*(t)] - \inf_{z \in (z_h(t_0) + V_H) \cap \cQ} \int_\Omega^{(h)} c[z] \leq {\nu |\Omega| \over t}.
\end{align}
\begin{proof}
\begin{align}
\int_\Omega^{(h)} c[z_{h,t_0,H}^*(t) - w] 
& = {1 \over t}\int_{\Omega}^{(h)} F'(Dz_{h,t_0,H}^*(t))[D(w-z_{h,t_0,H}^*(t))] \\ 
& \leq {1 \over t} \int_{\Omega}^{(h)} \nu = {\nu |\Omega| \over t},
\end{align}
where we have used \eqref{e:weak central path} and (NN2.3.2).
The result follows by taking an infimum over admissible $w(x)$.
\end{proof}
\end{lemma}

\begin{lemma}
Let $z \in \cQ \cap L^{\infty}$ and $q \in L^{\infty}$ and $h \geq 0$. Then,
\begin{align}
(\cF'_h(z)[q])^2 & \leq |\Omega| \nu \cF''_h(z)[q^2].
\end{align}
\end{lemma}
\begin{proof}
\begin{align}
|\cF'_h(z)[q]| & = \left|\int_{\Omega}^{(h)} F'(z)[q]\right| \leq \int_{\Omega}^{(h)} \sqrt{\nu F''(z)[q^2]}.
\end{align}
The conclusion follows from Jensen's inequality.
\end{proof}

If we also had
\begin{align}
|\cF'''(z)[q^3]| & \leq C \cF''(z)[q^2]^{1.5}, \label{e:third derivative}
\end{align}
for all $z \in \cQ$ and $q \in L^{\infty}$, then we would indeed have a self-concordant barrier on $\cQ$, which we have already noted is impossible in infinite dimensions. We must therefore find some relaxation of \eqref{e:third derivative}. We begin with a crude estimate that is useful as a ``fallback''.

\begin{lemma} \label{l:numerical barrier}
Let $C_1,\rho$ be as per \eqref{e:quadrature weights}.
The function $C_1^{-1} h^{-d}\cF_h$ is a self-concordant barrier for $\cQ_h$ with parameter 
\begin{align}
\nu(h) = O(h^{-d}). \label{e:crude estimate}
\end{align}
\end{lemma}
\begin{proof}
\begin{align}
C_1^{-1}h^{-d}\cF_h(w) = \sum_{\substack{K \in T_h \\ k=1,\ldots,k}} C_1^{-1}h^{-d}\omega_{K,k} F(w(x_{K,k})).
\end{align}
The coefficients $C_1^{-1}h^{-d}\omega_{K,k}$ of the sum are bounded between $1$ and $C_2/C_1$, so the result follows by self-concordant calculus.
\end{proof}
Any algorithm that relies on the estimate \eqref{e:crude estimate} typically results in $\tilde{O}(h^{-0.5d}) = \tilde{O}(\sqrt{n})$ damped Newton iterations when short $t$ steps are used. We now discuss a more nuanced theory with better iteration counts for some situations.

\begin{definition}[Regularity hypotheses] \label{d:regular}
Denote $\partial_{t^{-1}} = {\partial \over \partial t^{-1}} = t {\partial \over \partial t} = t \partial_t$.
Assume $T_h$ is a quasi-uniform triangulation of $\Omega$.
Denote by $\Pi_h$ the interpolation operator for the piecewise polynomial space $V_h \subset W^{1,\infty}(\Omega) \times L^{\infty}(\Omega)$; assume $DV_h$ is of degree $\alpha-1$, and that $\alpha \geq d$.
We say that $(T_h,c,F)$ is regular if the following properties are satisfied.
\begin{enumerate}
\item $F$ is a self-concordant barrier for $Q$ with parameter $\nu$, and $|\Lambda(q)| \geq \alpha_1 \|q\|_2 + \alpha_2$ for some constants $\alpha_1>0$ and $\alpha_2 \in \mathbb{R}$.
\item The uniform discrete reverse Hölder inequality. There is a function $C_{RH}(z_h^*(t))$ of $(h,t)$ such that, for all $0 < h \leq H$, $K \in T_H$, $t_0 \leq t < \infty$, polynomial $q$ such that $Dq$ has degree $\alpha-1$, then
\begin{align}
& \left\|\sqrt{F''(Dz_{h}^*(t))[(Dq)^2]}\right\|_{L^{\infty}_h(K)} \\
& \leq
C_{RH}(z_{h}^*(t))|K|^{-1}
\left\|\sqrt{F''(Dz_{h}^*(t))[(Dq)^2]}\right\|_{L^{1}_h(K)}. \nonumber
\end{align}
The function $C_{RH}(z_h^*(t))$ grows no faster than polylogarithmically in $(t,h)$.
\item Smoothness:
\begin{align}
D\partial_{t^{-1}} z_0^* \in L^{\infty}([t_0,\infty];W^{\alpha,\infty}(\Omega)).
\end{align}
\item Optimal approximation property:
\begin{align}
\|D z_0^* - D z_h^*\|_{L^{\infty}([t_0,\infty]\times\Omega)} 
+
 \|D\partial_{t^{-1}}z_0^* - D \partial_{t^{-1}}z_h^*\|_{L^{\infty}([t_0,\infty]\times\Omega)}
& \leq Ch^\alpha.
\end{align}
\end{enumerate}
\end{definition}

In complex analysis, if a function $f$ is analytic at a point $z^*$, then one may expand it into a power series. If $w$ is within the region of convergence of this power series, $f$ is also analytic at $w$, and one may find a new power series expansion, this time at $w$, to expand the domain of analyticity of $f$; this procedure is called ``analytic continuation''. It turns out that if $G(z(x))$ satisfies a certain reverse Hölder inequality for some $z \in \cQ_h$, then $G(w(x))$ will satisfy a related Hölder inequality if $w \in \cQ_h$ is in a suitable neighborhood of $z$, and this procedure can be used to propagate the reverse Hölder inequality to larger subsets of $\cQ_h$.

\begin{theorem}[Reverse Hölder continuation] \label{t:RH continuation}
Let $z^* \in \cQ_h$, $1 \geq H \geq h \geq 0$ and $t>0$. Consider the affine space $A = z^*+V_{H}$ and assume that $z^* = z^*_{A}(t)$ minimizes $z \to \cf_h(z,t)$ over $z \in A$. Assume that the following reverse Hölder inequality holds:
\begin{align}
\|\sqrt{F''(Dz^*)[(Dv)^2]}\|_{L^{\infty}_h(K)} & \leq C_{RH}(z^*) 
|K|^{-1}\|\sqrt{F''(Dz^*)[(Dv)^2]}\|_{L^{1}_h(K)}, 
\end{align}
for all $K \in T_{\hH}$, $1 \geq \hH \geq h$, and $v$ a polynomial such that $Dv$ has degree $\alpha-1$. Let $0 \leq \delta < 1$. Let $w \in A$ such that
\begin{align}
\cf_h(w,t) - \cf_h(z^*,t) \leq \delta (1+\sqrt{2|\Omega|})^{-2}C_{RH}^{-2}(z^*)C_{\min}^{2}H^{2d} =: \beta. \label{e:beta}
\end{align}
Then, the following reverse Hölder inequality also holds:
\begin{align}
\|\sqrt{F''(Dw)[(Dv)^2]}\|_{L^{\infty}_h(K)} & \leq C_{RH}(A,t,\beta) 
|K|^{-1}\|\sqrt{F''(Dw)[(Dv)^2]}\|_{L^{1}_h(K)} \\
\text{where } C_{RH}(A,t,\beta) & \leq C_{RH}(z^*)(1-\sqrt{\delta})^{-2}, \label{e:C_RH(w)}
\end{align}
for all $K \in T_{\hH}$, $1 \geq \hH \geq h$.
\end{theorem}
\begin{proof}
Note that
\begin{align}
\partial_w \cf_h(z^*,t) & = \int_{\Omega}^{(h)} tc[w] + F'(Dz^*)[Dw] = 0.
\end{align}
Therefore,
\begin{align}
\cf_h(w,t)  - \cf_h(z^*,t) & = \int_{\Omega}^{(h)} tc[w-z^*] + F(Dw) - F(Dz^*) \\
& = \int_{\Omega}^{(h)} F(Dw) - F(Dz^*) - F'(Dz^*)[D(w-z^*)].
\end{align}
We apply \eqref{e:sc norm ineq 1} with $G = F$ and $y = Dw$ and $z = Dz^*$ to arrive at
\begin{align}
\cf_h(w,t)  - \cf_h(z^*,t) 
& \geq \int_{\Omega}^{(h)} \psi\left(
\sqrt{F''(Dz^*)[(Dw-Dz^*)^2]}
\right) \\
& \geq |\Omega| \psi\left(
|\Omega|^{-1} \int_{\Omega}^{(h)}
\sqrt{F''(Dz^*)[(Dw-Dz^*)^2]}
\right),
\end{align}
where we have used Jensen's inequality.
We continue by using the bound
$\psi^{-1}(\beta) \leq \beta + \sqrt{2\beta}$ to arrive at
\begin{align}
\beta + \sqrt{2|\Omega|\beta} & \geq 
\int_{\Omega}^{(h)}
\sqrt{F''(Dz^*)[(Dw-Dz^*)^2]} \\
& = \sum_{K \in T_H}\|\sqrt{F''(Dz^*)[(Dw-Dz^*)^2]}\|_{L^1_h(K)} \\
& \geq \sum_{K \in T_H}C_{RH}^{-1}(z^*)|K|\|\sqrt{F''(Dz^*)[(Dw-Dz^*)^2]}\|_{L^\infty_h(K)} \\
& \geq C_{RH}^{-1}(z^*)\max_{K \in T_H}|K|\|\sqrt{F''(Dz^*)[(Dw-Dz^*)^2]}\|_{L^\infty_h(K)}.
\end{align}
\begin{align}
\left\|\sqrt{F''(Dz^*)[(Dw-Dz^*)^2]}\right\|_{L^{\infty}_h(K)}
& \leq (1+\sqrt{2|\Omega|})C_{RH}|K|^{-1}\sqrt{\beta} =: r.
\end{align}
provided $\beta \leq 1$.
Then, from \eqref{e:Hessians with r}, if $r<1$, we find that on $K \in T_{\hH}$,
\begin{align}
(1-r)^2 F''(Dz^*) \preceq F''(Dw) \preceq (1-r)^{-2}F''(Dz^*).
\end{align}
In particular, for any polynomial $v$ such that $Dv$ is of degree $\alpha-1$, we have the following reverse Hölder inequality:
\begin{align}
\|\sqrt{F''(Dw)[(Dv)^2]}\|_{L^{\infty}_h(K)} 
& \leq (1-r)^{-1}\|\sqrt{F''(Dz^*)[(Dv)^2]}\|_{L^{\infty}_h(K)} \\
& \leq (1-r)^{-1}C_{RH}|K|^{-1}\|\sqrt{F''(Dz^*)[(Dv)^2]}\|_{L^{1}_h(K)} \\
& \leq C_{RH}(1-r)^{-2}|K|^{-1}\|\sqrt{F''(Dw)[(Dv)^2]}\|_{L^{1}_h(K)},
\end{align}
valid for any $K \in T_{\hH}$.
\end{proof}

\begin{definition}
Let $z^* \in \cQ$, $\beta \geq 0$ and $H\geq h \geq 0$ and $t>0$. Consider the affine space $A = z^*+V_H$ and assume that $z^*$ minimizes $z \to \cf_h(z,t)$ over $z \in A$. Define the ``Lebesgue set''
\begin{align}
\cL_{A,t}(\beta) & = \{
w \in A \; : \; \cf_h(w,t) - \cf_h(z^*,t) \leq \beta
\}.
\end{align}
\end{definition}

Theorem \ref{t:RH continuation} states that reverse Hölderness propagates from a single point $z^*$ to a ``neighborhood'' $\cL_{A,t}(\beta)$. In Section 6, we shall use this continuation procedure iteratively.

\section{Analysis of the naïve algorithm}

\begin{lemma}
Assume $(T_h,c,F)$ is regular.
\begin{align}
\cf_h(z_h^*(t_1),t) - \cf_h(z_h^*(t),t) & \leq \nu |\Omega| (\rho - \log\rho - 1) \text{ where } \rho = {t \over t_1} \geq 1. \label{e:naive gap 1} \\
\cf_h(z_h^*(t),t) - \cf_h(z_0^*(t),t) & \leq C\min\{th^\alpha,(th^\alpha)^2\}. \label{e:naive gap 2}
\end{align}
\end{lemma}
\begin{proof}
From the regularity of $(T_h,c,F)$ and Lemmas \ref{l:basic self concordant estimates} and \ref{l:properties},
\begin{align}
\cf_h(z_h^*(t),t) - \cf_h(z_0^*(t),t) 
& \leq \int_{\Omega}^{(h)} \psi\left(-\sqrt{F''(z_0^*(t))[(Dz_h^*(t)-Dz_0^*(t))^2]}\right) \\
& \leq \int_{\Omega}^{(h)} \psi\left(-Ct\|Dz_h^*(t)-Dz_0^*(t)\|_2\right) \\
& \leq \int_{\Omega}^{(h)} \psi\left(-Cth^\alpha\right) \\
& \leq Ct^2h^{2\alpha},
\end{align}
valid for $0 \leq Cth^\alpha \leq 0.5$, and
where we have used that $\psi(\alpha) = O(\alpha^2)$ as $\alpha \to 0$ and $\psi$ is monotonically decreasing for $\alpha \leq 0$.

In the regime $Cth^\alpha > 0.5$, we use instead the following argument. Put $g(t) = \cf_h(z_h^*(t),t)-\cf_h(z_0^*(t),t)$. Then,
\begin{align}
g'(t) & = \int_{\Omega}^{(h)} c[z_h^*(t)-z_0^*(t)] \, dx \\
& + \overbrace{\int_{\Omega}^{(h)} tc[\partial_t z_h^*(t)] + F'(Dz_h^*(t))[D\partial_t z_h^*(t)] \, dx}^{0} \\
& - \overbrace{\int_{\Omega}^{(h)} tc[\partial_t z_0^*(t)] + F'(Dz_0^*(t))[D\partial_t z_0^*(t)] \, dx}^{0}.
\end{align}
The regularity of $(\Omega,c,F)$ then gives $|g'(t)| \leq Ch^\alpha$ and hence $g(t) \leq g(1) + (t-t_1)Ch^\alpha \leq tCh^\alpha + O(1)$, which proves \eqref{e:naive gap 2}.

Now set $g(t) = \cf_h(z^*_h(t_1),t) - \cf_h(z^*_h(t),t)$. Note that $g(t_1)=0$. Furthermore,
\begin{align}
g'(t) & = \int_{\Omega}^{(h)} c[z^*_h(t_1)-z^*_h(t)] \, dx \\
& - \overbrace{\int_{\Omega}^{(h)} tc[\partial_t z^*_h(t)] + F'(Dz^*_h(t))[D\partial_t z^*_h(t)] \, dx}^{0}.
\end{align}
We see that $g'(t_1) = 0$. Thus,
\begin{align}
g''(t) & = -\int_{\Omega}^{(h)} c[\partial_t z_{h}^*(t)] \\
& = {1 \over t}\int_{\Omega}^{(h)} F'(D z_{h}^*)[D\partial_tz_{h}^*] \\
& \leq {\sqrt{\nu} \over t} \int_{\Omega}^{(h)} \sqrt{F''[(D\partial_t z_{h}^*)^2]} \\
& \leq {\sqrt{\nu|\Omega|} \over t}\sqrt{
\int_{\Omega}^{(h)} F''[(D\partial_tz_{h}^*)^2]
} \\
& = {\sqrt{\nu|\Omega|} \over t}\sqrt{-\int_{\Omega}^{(h)} c[\partial_t z_{h}^*(t)}] \\
& = {\sqrt{\nu|\Omega|} \over t}\sqrt{g''(t)}
\end{align}
Thus,
\begin{align}
g''(t) & \leq {\nu |\Omega| \over t^2}.
\end{align}
The result follows by integrating twice.
\end{proof}

We now prove our first main theorem.

\begin{proof}[Proof of Theorem \ref{t:naive iteration count}]
In view of Lemma \ref{l:numerical barrier}, note that short $t$ steps on the fine grid satisfy $t_{k+1} = \rho t_k$ with $\rho-1 \sim {h^{0.5d}}$.

We begin with the analysis of the $h$-then-$t$ schedule.
The initial step of the algorithm is to start from an admissible $z^{(0)} \in V_{h^{(1)}} \cap \cQ$, the coarsest space, and find the center $z^*_{h^{(1)}}(t_0)$ by damped Newton iterations. This will require a certain number $N_0$ of damped Newton iterations, but this initial problem the same regardless of the choice of the finest grid level. In other words, $N_0$ is independent of the fine grid parameter $h$, so this initial step requires $O(1)$ damped Newton iterations.

According to Lemma \ref{l:numerical barrier} and \eqref{e:naive gap 2}, since $t_0 = O(1)$, for any grid level $h^{(\ell)} \geq h^{(L)} = h$, the function $Ch^{-d}\cf_h(w,t_0)$ is standard self-concordant on $V_{h^{(\ell)}} \cap \cQ$, and the suboptimality gap is $O(h^{\alpha-d}) = O(1)$ so each $h$ refinement converges in $O(1+\log\log\epsilon^{-1})$ damped Newton iterations.

Once on the fine grid, according to \eqref{e:naive gap 1}, the short $t$ step length is optimal, resulting in $\tilde{O}(h^{-0.5d})$ damped Newton iterations.

Now consider an arbitrary schedule of $t$ and $h$ refinements. We only consider the final grid refinement (i.e. from level $h^{(L-1)}$ to $h^{(L)} = h$), and the subsequent $t$ refinement on the fine grid $h$. Say that this occurs at iteration $j$, i.e. $h_j = h^{(L-1)}$ and $h_{j+1} = h^{(L)} = h$. We make two cases. First, if $t_j > h^{-\alpha}$, then the suboptimality gap of $Ch^{-d}\cf$ for the final $h$ refinement given by \eqref{e:naive gap 2} is at best $O(h^{-d})$. By the standard theory, short $t$ steps are theoretically optimal and converge in $\tilde{O}(h^{-0.5d})$ damped Newton iterations.

We now consider the case $t_j \leq h^{-\alpha}$. We count the $t$ refinements on the fine grid. Because short $t$ steps are optimal, and because the stopping criterion is $t \sim h^{-2\alpha}$, the theoretical estimate must be at least $\tilde{O}(h^{-0.5d})$ damped Newton iterations for these $t$ refinements.
\end{proof}

\section{Analysis of Algorithm MGB}

\begin{lemma} \label{l:mg h refinement}
Assume $(T_h,c,F)$ is regular.
There is a constant $\Chref$ such that
\begin{align}
\cf_h(z_{h,t_1,H}^*(t),t) - \cf_h(z_h^*(t),t) \leq \Chref^2 H^{2\alpha}\left({t \over t_1} - 1\right)^2, \label{e:mg h refinement}
\end{align}
provided that $C_{\href}(t/t_1-1) \leq 0.6838$.
\end{lemma}
\begin{proof}
We shall denote by $\Pi_h$ the interpolation operator for $V_h$.
Let
\begin{align}
w_{h,t_1,H}(t) & = z_h^*(t_1) + \Pi_H (z_0^*(t)-z_0^*(t_1)).
\end{align}
Put $g(t) = Dw_{h,t_1,H}(t) - Dz_h^*(t)$, and note that $g(t_1) = 0$. Then,
\begin{align}
\|Dw_{h,t_1,H}(t) - Dz_h^*(t)\|_{L^{\infty}_h} & = \|g(t)\|_{L^{\infty}_h} =
\|g(t) - g(t_1)\|_{L^{\infty}_h} \\
& = \left\|
\int_{t_1^{-1}}^{t^{-1}} \partial_{t^{-1}} g(\tau) \, d\tau
\right\|_{L^{\infty}_h} \\
& \leq \int_{t_1^{-1}}^{t^{-1}} \left\|
D\Pi_H \partial_{t^{-1}}z_0^*(\tau) - D\partial_{t^{-1}} z_h^*(\tau)
\right\|_{L^{\infty}_h} \, d\tau \\
& \leq \int_{t_1^{-1}}^{t^{-1}} \left\|
D\Pi_H \partial_{t^{-1}}z_0^*(\tau) - D\partial_{t^{-1}} z_0^*(\tau)
\right\|_{L^{\infty}_h} \\
& \qquad +\left\|
D\partial_{t^{-1}}z_0^*(\tau) - D\partial_{t^{-1}} z_h^*(\tau)
\right\|_{L^{\infty}_h} d\tau \\
& \leq CH^\alpha(t_1^{-1}-t^{-1}).
\end{align}
Thus,
\begin{align}
\cf_h(w_{h,t_1,H}(t)) - \cf_h(z_h^*(t)) &
\leq \int_{\Omega}^{(h)} \psi\left(-\sqrt{F''(Dz_h^*)[(Dw_{h,t_1,H}(t) - Dz_h^*(t))^2]}\right) \\
& \leq \int_{\Omega}^{(h)} \psi\left(-Ct\|Dw_{h,t_1,H}(t) - Dz_h^*(t)\|_2\right) \\
& \leq \int_{\Omega}^{(h)} \psi\left(-CH^\alpha(t/t_1-1)\right) \\
& \leq CH^{2\alpha}(t/t_1-1)^2,
\end{align}
where we have used that $\psi(\alpha) \leq \alpha^2$ when $-0.6838 \leq \alpha \leq 0$ and $\psi(\alpha)$ is monotonically decreasing for $\alpha \leq 0$.
\end{proof}

\begin{lemma}
Denote by $L = O(\log h)$ the number of grid levels, from the coarsest level $h^{(1)}$ to the fine grid level $h = h^{(L)}$. Assume that $\alpha \geq d$. We denote $z_h^* = z_h^*(t)$ (i.e. the ommited $(t)$ is implied), but $z_h^*(t_1)$ has its usual meaning.
For $t > 0$,
assume that $z_h^* = z_h^*(t)$ satisfies the following reverse Hölder inequality:
\begin{align}
\|\sqrt{F''(Dz_h^*)[(Dv)^2]}\|_{L_h^\infty(K)} 
& \leq C_{RH}(z_h^*) |K|^{-1}
\|\sqrt{F''(Dz_h^*)[(Dv)^2]}\|_{L_h^1(K)},
\end{align}
for all $1 \geq H \geq h$, $K \in T_H$ and polynomial $v$ such that $Dv$ is of degree $\alpha - 1$.
Define 
\begin{align}
\tilde{\beta} & = 4^{-d}e^{-4}(1+\sqrt{2|\Omega|})^{-2}C_{\min}^2 \\
\beta & = (L+1)^{-2} {C}_{RH}^{-2}(z_h^*(t)) \tilde{\beta}.
\end{align}
Denote $\rho = t/t_1$ and assume that
\begin{align}
\rho & < 1 + \Chref^{-1}{\sqrt{{\beta}}
}.
\end{align}
For $\ell = 1,\ldots,L$, put $A_{\ell} = z_h^*(t_1) + V_{h^{(\ell)}}$. Then,
\begin{align}
\cf_h(z_{h,t_1,h^{(\ell)}}^*) - \cf_h(z_{h}^*) 
& \leq 
(0.5 h^{(\ell)})^{2d} {\beta}
\text{ and}\\
C_{RH}(A_{\ell},t,(h^{(\ell)})^{2d}{\beta})
& \leq e^2C_{RH}(z_h^*).
\end{align}
for all $K \in T_H$ and $1 \geq H \geq h$ and polynomial $v$ such that the degree of $Dv$ is $\alpha - 1$.

Furthermore, if
\begin{align}
w \in \cL_{A_\ell,t}((h^{(\ell)})^{2d}{\beta}),
\end{align}
then, for an arbitrary test function $\phi \in V_H$,
\begin{align}
\left|\cF'''_h(Dw)[(D\phi)^3] \right|
& \leq 2e^2C_{RH}(z_h^*)C_{\min}^{-0.5}H^{-0.5d}(\cF(Dw)[(D\phi)^2])^{1.5}.
\end{align}
In particular, the function $w \to e^4 C_{RH}^2(z_h^*)C_{\min}^{-1}(h^{(\ell)})^{-d}\cf_h(w,t)$ is standard self-concordant on $\cL_{A_{\ell},t}((h^{(\ell)})^{2d}\beta)$ with suboptimality gap bounded by 
\begin{align}
4^{-d}(h^{(\ell)})^{(d)}C_{\min}(1+\sqrt{2|\Omega|})^{-2}(L+1)^{-2}.
\end{align}
The damped Newton method on $\cL_{A_{\ell},t}((h^{(\ell)})^{2d}\beta)$ converges in
\begin{align}
O(1)+\log\log\epsilon^{-1} \text{ iterations.} \label{e:iterations argh}
\end{align}
\end{lemma}
\begin{proof}
For $\ell = 1,\ldots,L$, we begin by proving a reverse Hölder inequality of the form
\begin{align}
 \|\sqrt{F''(Dz_{h,t_1,h^{(\ell)}}^*)[(Dv)^2]}\|_{L^\infty_h(K)} 
& \leq C
|K|^{-1}
\|\sqrt{F''(Dz_{h,t_1,h^{(\ell)}}^*)[(Dv)^2]}\|_{L^1_h(K)}, \label{e:RH yet again}
\end{align}
for all $1 \geq H \geq h > 0$ and for all $K \in T_{H}$.
We shall denote by $C_{RH}(z_{h,t_1,h^{(\ell)}}^*)$ the smallest constant $C$ such that \eqref{e:RH yet again} holds.
We do a proof by induction ``backwards'', starting from $\ell = L$, that the following inequality holds
\begin{align}
{C_{RH}(z_{h,t_1,h^{(\ell)}}^*)} & \leq \left(1 - {1 \over L+1}\right)^{-2(L-\ell)}C_{RH}(z_h^*). \label{e:C_RH induction}
\end{align}
For $\ell = L$, since $z_{h,t_1,h^{(L)}}^* = z_{h,t_1,h}^* = z_h^*$, the induction hypothesis is tautological. We now prove by induction the cases $\ell = L-1,\ldots,1$. We find that
\begin{align}
& \cf_h(z_{h,t_1,h^{(\ell)}}^*) - \cf_h(z_{h,t_1,0.5h^{(\ell)}}^*) \\
& \leq 
\cf_h(z_{h,t_1,h^{(\ell)}}^*) - \cf_h(z_{h}^*) \\
& \stackrel{\eqref{e:mg h refinement}}{\leq} \Chref^2 (h^{(\ell)})^{2d}(\rho-1)^2 \\
& \leq (h^{(\ell)})^{2d}(L+1)^{-2}{C}_{RH}^{-2}(z_h^*(t))\tilde{\beta} \\
& = (h^{(\ell)})^{2d}(L+1)^{-2} 4^{-d}e^{-4}(1+\sqrt{2|\Omega|})^{-2}{C}_{RH}^{-2}(z_h^*)C_{\min}^2.
\end{align}
Note that $\left(1 - {1 \over L+1}\right)^{-2(L-\ell)} \leq \left(1 - {1 \over L+1}\right)^{-2L} \leq e^2$, so from \eqref{e:C_RH induction} with $\ell$ replaced by $\ell+1$ (i.e. the induction hypothesis), we find that $C_{RH}(z_{h,t_1,h^{(\ell+1)}}^*) \leq e^2C_{RH}(z_h^*)$, so that
\begin{align}
\cf_h(z_{h,t_1,h^{(\ell)}}^*) - \cf_h(z_{h,t_1,0.5h^{(\ell)}}^*) 
& \leq \beta_{\ell+1} \text{ where}\\
\beta_{\ell} & = {C_{\min}^2 (h^{(\ell)})^{2d} \over 
(L+1)^2(1+\sqrt{2|\Omega|})^2C_{RH}^2(z_{h,t_1,h^{(\ell)}}^*)
}.
\end{align}
We put $A = A_{\ell+1} = z^*_h(t_1) + V_{0.5h^{(\ell)}}$ and $\beta = \beta_{\ell+1}$ to find that \eqref{e:beta} holds with $\delta = (L+1)^{-2} < 1$. Thus,
\begin{align}
C_{RH}(A_{\ell+1},t,\beta_{\ell+1})
& \stackrel{\eqref{e:C_RH(w)}}{\leq} 
(1-\sqrt{\delta})^{-2}C_{RH}(z_{h,t_1,0.5h^{(\ell)}}^*) \\
& \stackrel{\eqref{e:C_RH induction}}{\leq}
\left(
1 - {1 \over L+1}
\right)^{-2}
\left(
1 - {1 \over L+1}
\right)^{-2(L-\ell-1)}C_{RH}(z_h^*) \\
& = \left(
1 - {1 \over L+1}
\right)^{-2(L-\ell)}C_{RH}(z_h^*).
\end{align}
Then, from $z_{h,t_1,h^{\ell}}^* \in \cL_{A_{\ell+1},t}(\beta_{\ell+1})$, we have that $C_{RH}(z_{h,t_1,h^{\ell}}^*) \leq C_{RH}(A_{\ell+1},t,\beta_{\ell+1})$, and there follows \eqref{e:C_RH induction}. This completes the induction proof of \eqref{e:C_RH induction} for $\ell = 1,\ldots,L$.

We now prove the self-concordance of $\cF_h$. Let $\phi \in V_H$ be an arbitrary test function, and write $\phi = \sum_{K \in T_H} \mathbbm{1}_K v_K$, where each $v_K$ is a polynomial such that the degree of $Dv_k$ is $\alpha-1$.
\begin{align}
& \left|\cF'''_h(Dw)[(D\phi)^3]\right| \\
& \leq \int_{\Omega}^{(h)} 2|F''(Dw)[(D\phi)^2]|^{1.5} \\
& \leq 2\|F''(Dw)[(D\phi)^2]\|_{L^1_h(\Omega)}
\|\sqrt{F''(Dw)[(D\phi)^2]}\|_{L^\infty_h(\Omega)} \\
& = 2\|F''(Dw)[(D\phi)^2]\|_{L^1_h(\Omega)}
\max_{K \in T_H}
\|\sqrt{F''(Dw)[(Dv_K)^2]}\|_{L^\infty_h(K)} \\
& \leq 2e^2C_{RH}(z_h^*)\|F''(Dw)[(D\phi)^2]\|_{L^1_h(\Omega)}
\max_{K \in T_H}|K|^{-1}\|\sqrt{F''(Dw)[(Dv_K)^2]}\|_{L^1_h(K)} \\
& \leq 2e^2C_{RH}(z_h^*)\|F''(Dw)[(D\phi)^2]\|_{L^1_h(\Omega)}
\max_{K \in T_H}|K|^{-0.5}\sqrt{\|F''(Dw)[(Dv_K)^2]\|_{L^1_h(K)}} \\
& \leq 2e^2C_{RH}(z_h^*)C_{\min}^{-0.5}H^{-0.5d}\|F''(Dw)[(D\phi)^2]\|_{L^1_h(\Omega)}^{1.5}.
\end{align}
\end{proof}

\begin{lemma} \label{l:mg t refinement}
\begin{align}
\cf_h(z_{h}^*(t_1),t) - \cf_h(z_{h,t_1,H}^*(t),t)
& \leq \nu|\Omega|(\rho-\log \rho - 1) \text{ where } \rho = {t \over t_1} \geq 1. \label{e:mg t refinement}
\end{align}
\end{lemma}
\begin{proof}
Let $g(t) = \cf_h(z_{h}^*(t_1),t) - \cf_h(z_{h,t_1,H}^*(t),t)$.
\begin{align}
g'(t) & = \int_{\Omega}^{(h)} c[z_h^*(t_1)-z_{h,t_1,H}^*(t)] \, dx \\
& - \overbrace{\left(
\int_\Omega^{(h)} tc[\partial_t z_{h,t_1,H}^*(t) + F'(Dz_{h,t_1,H}^*(t))[D\partial_t z_{h,t_1,H}^*(t)]
\, dx
\right)}^{0},
\end{align}
where we have used that $\partial_t z_{h,t_1,H}^*(t) \in V_H$, the tangent space of $A = z_h^*(t_1) + V_H$. We further see that $g'(t_1) = 0$. Thus,
\begin{align}
|g''(t)| & = \left|\int_\Omega^{(h)} c[\partial_t z_{h,t_1,H}^*(t)] \right|\\
& = {1 \over t}\left|\int_{\Omega}^{(h)}
F'(Dz_{h,t_1,H}^*(t))[D\partial_tz_{h,t_1,H}^*(t)]
\right| \\
& \leq {1 \over t} \int_\Omega^{(h)} \sqrt{\nu F''(Dz_{h,t_1,H}^*(t))[(D\partial_tz_{h,t_1,H}^*(t))^2]} \\
& \leq {1 \over t} \sqrt{\nu|\Omega|\int_\Omega^{(h)} F''(Dz_{h,t_1,H}^*(t))[(D\partial_tz_{h,t_1,H}^*(t))^2]} \\
& = {1 \over t} \sqrt{-\nu|\Omega|\int_\Omega^{(h)} c[\partial_tz_{h,t_1,H}^*(t)]} \\
& = {1 \over t} \sqrt{\nu|\Omega|\,|g''(t)|}.
\end{align}
As a result,
\begin{align}
|g''(t)| & \leq {\nu|\Omega| \over t^2}.
\end{align}
\end{proof}

\begin{lemma}
Define
\begin{align}
\Cmg & = \min \left\{
{\sqrt{\tilde{\beta}} \over \Chref}, \; 
{
\sqrt{2\tilde{\beta}}\,(h^{(1)})^d
\over
\sqrt{\nu|\Omega|}
}
\right\}. \label{e:Cmg}
\end{align}
Assume
\begin{align}
\rho & \leq 1 + 
{\Cmg 
\over
(L+1)C_{RH}(z_h^*(t))
}. \label{e:rho mg}
\end{align}
Put $t = \rho t_1$.
Then, Algorithm MGB compute $z_h^*(t)$ from $z_h^*(t_1)$ in 
\begin{align}
L\left(O(1) + \log \log \epsilon^{-1}\right) \text{ Newton iterations.} \label{e:neverending iterations}
\end{align}
\end{lemma}
\begin{proof}
Denote $A_\ell = z_h^*(t_1) + V_{h^{(\ell)}}$. From \eqref{e:mg t refinement}, \eqref{e:Cmg}, \eqref{e:rho mg} and from $\rho - \log\rho - 1 \leq 0.5(\rho-1)^2$,
\begin{align}
z_h^*(t_1) \in \cL_{A_1,t}(\tilde{\beta}(h^{(1)})^{2d}(L+1)^{-2}C_{RH}^{-2}(z_h^*(t))) = \cL_{A_1,t}(\beta(h^{(1)})^{2d}).
\end{align}
According to \eqref{e:iterations argh}, the damped Newton methos starting at $z_h^*(t_1)$ will locate $z_{h,t_1,h^{(1)}}^*(t)$ in
\begin{align}
O(1) + \log \log \epsilon^{-1} \text{ iterations}. \label{e:aaaargh}
\end{align}
Furthermore, for each $\ell = 1,\ldots,L-1$, we find that
\begin{align}
z_{h,t_1,h^{(\ell)}}^*(t) \in \cL_{A_{\ell+1},t}((h^{(\ell+1)})^{2d}\beta).
\end{align}
According to \eqref{e:iterations argh}, the damped Newton methos starting at $z_{h,t_1,h^{(\ell)}}^*(t)$ will locate $z_{h,t_1,h^{(\ell+1)}}^*(t)$ in \eqref{e:aaaargh} iterations. Thus, the total number of iterations is obtained by multiplying \eqref{e:aaaargh} by $L$.
\end{proof}

We are now ready to prove our second main theorem.

\begin{proof}[Proof of Theorem \ref{t:Algorithm MGB}]
The initial step of Algorithm MGB begins with an admissible $z^{(0)} \in V_{h^{(1)}} \cap \cQ$, finds $z^*_{h^{(1)}}(t_0)$ by damped Newton steps, and from there performs $h$ refinements to compute $z^*_h(t_0)$. This procedure is identical to the initial phase of the $h$-then-$t$ schedule of the naïve algorithm. The proof of Theorem \ref{t:naive iteration count} shows that this initial phase requires $\tilde{O}(1)$ damped Newton steps.

The number of damped Newton iterations to compute $z_h^*(t_{k+1})$ from $z_h^*(t_k)$ is given by \eqref{e:neverending iterations}. It thus suffices to count the number of $t$-steps. The $t$ step size is limited by \eqref{e:rho mg}. Since we start with $t_0 = O(1)$ and end when $t_k \sim h^{2\alpha}$, the total number of $t$ steps at most
\begin{align}
O\left(
{(\log(h^{2\alpha}/t_0))
(1-\log_2(h))\sup_{t \in [t_0,h^{2\alpha}]} C_{RH}(z_h^*(t))
}
\right).
\end{align}
This whole expression is bounded by a polylogarithm of $h$.
\end{proof}

\section{Reverse inequalities}

In the present section, we show that many functions satisfy reverse Hölder and Sobolev inequalities. Our goal is to show that the reverse Hölder inequality of Definition \ref{d:regular} is satisfied if the solution $z^*_0$ and the barrier $F$ satisfy some smoothness conditions.

\begin{lemma}
Let $D(x,R) \subset \mathbb{C}$ be a disc centered at $x \in \mathbb{C}$ of radius $R$, and let $r<R$. Then, for any bounded analytic function $f(x)$ on $D(x,R)$,
\begin{align}
\|f'\|_{L^{\infty}(D(x,r))} & \leq {R \over (R-r)^2} \|f\|_{L^{\infty}(D(x,R))}. \label{e:Cauchy}
\end{align}
\end{lemma}
\begin{proof}
For $y \in D(x,r)$, the Cauchy integral formula gives
\begin{align}
|f'(y)| & = \left|{1 \over 2\pi i}\oint_{\partial D(x,R)} {f(z) \over (z-y)^2} \, dz\right| \\
& \leq {R \over (R-r)^2} \|f\|_{L^{\infty}(\partial D(x,R)}.
\end{align}
\end{proof}

\begin{lemma}
Denote by $P_{\beta}$ the set of polynomials in $x \in \mathbb{C}$ of degree $\beta$, and let $\epsilon>0$. Given $\beta$, there is a constant $C = C(\beta,\epsilon)$ such that the following holds. For any $q \in P_\beta$, $x \in \mathbb{R}$ and $r>0$,
\begin{align}
\|q\|_{L^\infty(D(x,r))} & \leq Cr^{-1}\|q\|_{L^1(x-\epsilon r,x+\epsilon r)}. \label{e:polynomials are special}
\end{align}
\end{lemma}
\begin{proof}
For $u \in P_\beta$, by norm equivalence in finite dimensions, there is a constant $C$ such that
\begin{align}
\|u\|_{L^{\infty}(D(0,1))} \leq C\|u\|_{L^1(-\epsilon,\epsilon)}.
\end{align}
For arbitrary $q \in P_\beta$, the substitution $u(y) = q((y-x)/r)$ gives
\begin{align}
\|q\|_{L^{\infty}(D(x,r))} & = \|u\|_{L^{\infty}(D(0,1))} \\
& \leq C\|u\|_{L^1(-\epsilon,\epsilon)} \\
& = Cr^{-1}\|q\|_{L^1(x-\epsilon r,x+\epsilon r)}.
\end{align}
\end{proof}

Let $U \subset \mathbb{C}^{m}$ be a domain. We now recall the Weierstrass preparation theorem, see \citet[Theorem 6.4.5]{krantz2001function} for details. A Weierstrass polynomial of degree $\beta$ is a function $p(x,y) = x^\beta + \sum_{j=0}^{\beta-1} a_j(y) x^j$ defined on some polydisc $(x,y) \subset B$, such that each function $a_j(y)$ is holomorphic on its domain. A unit $u(x,y)$ is a nonzero holomorphic function on $B$. The Weierstrass preparation theorem states that, if $f$ is holomorphic on $U$ and $z \in U$ then $f(z) = p(z)u(z)$ for some Weierstrass polynomial $p$, unit $u$, on some polydisc neighborhood $B \subset U$ of $z$.

The Weierstrass polynomial satisfies the formula
\begin{align}
p(x,y) & = \prod_{j=1}^{\beta} (x-\alpha_j(y)),
\end{align}
where $\alpha_1(y),\ldots,\alpha_{\beta}(y)$ are the roots of the function $x \to f(x,y)$ on $B$. The functions $\{\alpha_j(y)\}$ are holomorphic in $y$.
\begin{lemma} \label{l:special Weierstrass}
With the notation of the Weierstrass preparation theorem, if $f(x,y) \geq 0$ when $(x,y) \in \mathbb{R}^m$, then 
\begin{align}
p(x,y) & = \prod_{j=1}^{\beta/2} (x - \beta_j(y))(x - \bar{\beta}_j(y)), \label{e:special Weierstrass}
\end{align}
where 
 each $\beta_j$ satisfies $\Re \beta_j(y) \geq 0$.
\end{lemma}
\begin{proof}
Passing to a smaller neighborhood $B$ if necessary, $f(x,y)$ is given by its power series, which has real coefficients, and thus satisfies $f(\bar{x},y) = \widebar{f(x,y)}$. Thus, the roots of $f(x,y)$ in $B$ are either real (in which case they must be of even order because $f \geq 0$), or they must occur in conjugate pairs, giving rise to the roots $\beta_j(y)$ in \eqref{e:special Weierstrass}.
\end{proof}
\begin{corollary} \label{c:special Weierstrass}
We set 
\begin{align}
r(x,y) & = \prod_{j=1}^{\beta/2}(x-\beta_j(y)) \text{ and} \\
v(z) & = \sqrt{u(z)}.
\end{align}
Then, $f(x,y) = v^2(x,y)|r(x,y)|^2$ on $B \cap \mathbb{R}^m$, and $v$ is a unit on $B$.
\end{corollary}
\begin{proof}
The fact that $r(x,y)\overline{r(\bar{x},y)} = p$ is directly from Lemma \ref{l:special Weierstrass}. Passing to a smaller $B$ if necessary, since $u \neq 0$ on $B$, we may assume that $u(B)$ is contained in a half-plane that excludes the origin. Thus, $v = \sqrt{u}$ is a well-defined holomorphic unit function.
\end{proof}

\begin{lemma} \label{l:ball RS}
Let $z \in U \cap \mathbb{R}^m$, and let $0<\epsilon<1$. Assume $f$ is holomorphic on $U$, and $f \geq 0$ on $U \cap \mathbb{R}^m$. There is a constant $C$ and polydisc $B = \prod_j D(z_j,R_j)$such that the following holds. For every $a \in D(z_1,R_1) \cap \mathbb{R}$ and $\delta > 0$ such that $[a-2 \delta,a+2\delta] \subset D(z_1,R_1)$, and for every $y \in \mathbb{R}^{m-1} \cap \prod_{j=2}^m D(z_j,R_j)$,
\begin{align}
\delta\|\sqrt{f(\cdot,y)}\|_{L^{\infty}(a-\delta,a+\delta)}
+
\delta^2\|\sqrt{f(\cdot,y)}'\|_{L^{\infty}(a-\delta,a+\delta)}
\leq
C
\|\sqrt{f(\cdot,y)}\|_{L^{1}(a-\epsilon\delta,a+\epsilon\delta)} \label{e:ball RS}
\end{align}
where $\cdot'$ denotes the partial derivative with respect to $x$.
\end{lemma}
\begin{proof}
We use the Weierstrass preparation theorem, in the form of Corollary \ref{c:special Weierstrass}.
\begin{align}
\|\sqrt{f(\cdot,y)}\|_{L^{\infty}(a-\delta,a+\delta)}
& = \|vr\|_{L^{\infty}(a-\delta,a+\delta)} \\
& \leq \|v\|_{L^{\infty}} \|r(\cdot,y)\|_{L^{\infty}(D(a,\delta))} \\
& \stackrel{\eqref{e:polynomials are special}}{\leq} {C \over \delta} \|r(\cdot,y)\|_{L^{1}(a-\epsilon\delta,a+\epsilon\delta)} \\
& \leq {C \over \delta} \|f(\cdot,y)\|_{L^{1}(a-\epsilon\delta,a+\epsilon\delta)}.
\end{align}
Furthermore,
\begin{align}
\|\sqrt{f(\cdot,y)}'\|_{L^{\infty}(a-\delta,a+\delta)}
& = \|(v|r|)'\|_{L^{\infty}(a-\delta,a+\delta)} \\
& = \|v'|r| + v \sgn(r)r'\|_{L^{\infty}(a-\delta,a+\delta)} \\
& \leq C(\|r \|_{L^{\infty}(a-\delta,a+\delta)}
+
\|r'\|_{L^{\infty}(a-\delta,a+\delta)}) \\
& \stackrel{\eqref{e:Cauchy}}{\leq} C\left(\|r \|_{L^{\infty}(D(a,\delta))}
+
{2 \over \delta}
\|r\|_{L^{\infty}(D(a,2\delta))}\right) \\
& \stackrel{\eqref{e:polynomials are special}}{\leq} {C \over \delta^{2}}\|r \|_{L^{1}(a-\epsilon\delta,a+\epsilon\delta)} \\
& \leq {C \over \delta^{2}}\|\sqrt{f}\|_{L^{1}(a-\epsilon\delta,a+\epsilon\delta)}.
\end{align}
\end{proof}

The following inequality is also sometimes called a reverse Poincaré or Friedrichs inequality.
\begin{lemma}[Strong reverse Sobolev inequality] \label{l:strong RS}
Let $U \subset \mathbb{C}^m$ be a domain, and $K \subset U \cap \mathbb{R}^m$ be compact. For $z \in \mathbb{C}^m$, denote $z = (z^{(j)},\tilde{z}^{(j)})$ with $z^{(j)} \in \mathbb{C}^j$. Let $f$ be holomorphic on $U$ and $f \geq 0$ on $U \cap \mathbb{R}^m$ and $\epsilon>0$. There is a constant $C$ such that the following holds. If $z \in K$ and $\delta>0$, put $V^{(j)}(\delta) = \prod_{i=1}^j [z^{(j)}_i - \delta,z^{(j)}_i+\delta]$. If $V^{(j)}(2\delta) \times \{\tilde{z}^{(j)}\} \subset U$, then
\begin{align}
& \delta^j\|\sqrt{f(\cdot,\tilde{z}^{(j)})}\|_{L^{\infty}(V^{(j)}(\delta))}
+
\delta^{j+1}\|\partial_{z^{(j)}}\sqrt{f(\cdot,\tilde{z}^{(j)})}\|_{L^{\infty}(V^{(j)}(\delta))} \nonumber \\
& \leq C
\|\sqrt{f(\cdot,\tilde{z}^{(j)})}\|_{L^{1}(V^{(j)}(\epsilon \delta))}. \label{e:strong RS}
\end{align}
\end{lemma}
\begin{proof}
We can cover $K$ by polydiscs $B$ as per Lemma \ref{l:ball RS}, so that we may replace $U$ by some polydisc $B$. On $B$, we proceed by induction on $j$.

If $j=1$, then \eqref{e:strong RS} coincides with \eqref{e:ball RS}.

For the inductive step, assume that $\eqref{e:strong RS}$ holds for a given value of $j$, we show that it also holds with $j$ replaced by $j+1$.
\begin{align}
&
\delta^{j+1}\|\sqrt{f(\cdot,\tilde{z}^{(j+1)})}\|_{L^{\infty}(V^{(j+1)}(\delta))}
+
\delta^{j+2}\|\partial_{z_1}\sqrt{f(\cdot,\tilde{z}^{(j+1)})}\|_{L^{\infty}(V^{(j+1)}(\delta))} \\
& = \delta \sup_{\xi \in [z_{j+1}-\delta,z_{j+1}+\delta]} \left(
\delta^{j}\|\sqrt{f(\cdot,\xi,\tilde{z}^{(j+1)})}\|_{L^{\infty}(V^{(j)}(\delta))} \right. \\
& \qquad\qquad\qquad\qquad \left. +
\delta^{j+1}\|\partial_{z_1}\sqrt{f(\cdot,\xi,\tilde{z}^{(j+1)})}\|_{L^{\infty}(V^{(j)}(\delta))}\right) 
\\
& \stackrel{\eqref{e:strong RS}}{\leq} C\delta \sup_{\xi \in [z_{j+1}-\delta,z_{j+1}+\delta]}
\|\sqrt{f(\cdot,\xi,\tilde{z}^{(j+1)})}\|_{L^{1}(V^{(j)}(\delta))} \\
& \leq C\left\|\delta \sup_{\xi \in [z_{j+1}-\delta,z_{j+1}+\delta]} \sqrt{f(\cdot,\xi,\tilde{z}^{(j+1)})}\right\|_{L^{1}(V^{(j)}(\delta))} \\
& \stackrel{ \eqref{e:ball RS}}{\leq} C\left\| \int_{z_{j+1}-\delta}^{z_{j+1}+\delta} \sqrt{f(\cdot,\xi,\tilde{z}^{(j+1)})} \, d\xi \right\|_{L^{1}(V^{(j)}(\delta))} \\
& = C\|\sqrt{f(\cdot,\tilde{z}^{(j)})}\|_{L^{1}(V^{(j)}(\epsilon \delta))}.
\end{align}
Permuting the entries of $z$ if necessary, we see that the partial derivative $\partial_{z_1}$ can be replaced by any $\partial_{z_i}$ with $i=1,\ldots,j+1$, and the conclusion follows.
\end{proof}

\begin{lemma} \label{l:epsilon 0}
Assume that $U \times Y \subset \mathbb{C}^m$ is a domain, and $f^2 : U \times Y \to \mathbb{C}$ is complex analytic. Let $L \subset U \times Y \cap \mathbb{R}^m$ be compact. Assume $f(w) \geq 0$ for all $w \in L$. Denote $w = (x,y)$ with $x \in \mathbb{C}^d$. Assume $g_0 : \Omega \times Y \to L$.
Assume that the singular values of $\partial_x g_0$ are uniformly bounded above and below.
For $H \geq h \geq 0$, 
assume 
$g_h : \Omega \times Y \to L$
with
$\|g_h(\cdot,y) - g_0(\cdot,y)\|_{L^{\infty}(\Omega)} \leq C_0 h$, where $C_0>0$ is some constant.
There are constants $\epsilon_0>0$ and $C_1<\infty$ such that for any $H$ such that $h \leq \epsilon_0 H$ and $K \in T_H$, there holds the reverse Hölder inequality:
\begin{align}
\|f(g_h(\cdot,y),y)\|_{L^{\infty}_h(K)} 
& \leq C_1H^{-d}
\|f(g_h(\cdot,y),y)\|_{L^{1}_h(K)}. \label{e:hopefully near the end}
\end{align}
\end{lemma}
\begin{proof}
For any $K \in T_H$, let $x(K) \in K$ be the center of mass, and $z = z(K) = g_0(x(K))$. The bounds on the singular values of $\partial_x g_0$ guarantee that $V(\epsilon H) \subset g_0(K) \subset V(CH)$, where $V(\delta) = \prod_{i=1}^d [z_i-\delta,z_i+\delta]$, and $0 < \epsilon < C < \infty$ are constants. In addition, from $\int_K f(g_0(x,y),y) \, dx = \int_{g_0(K)} f(w)/\det ((\partial_x g_0)(g_0^{-1}(w))) \, dx$ and the bounds on $\partial_x g_0$, we have
\begin{align}
C_3 \|f(\cdot,y)\|_{L^1(V(\epsilon H))} 
\leq \|f(g_0(\cdot,y),y)\|_{L^1(K)}
\leq C_4 \|f(\cdot,y)\|_{L^1(V(C H))}. 
\end{align}
Thus,
\begin{align}
& \left| \int_K^{(h)} f(g_h(\cdot,y),y) - \int_K f(g_0(\cdot,y),y) \right| \\
& \leq \int_K^{(h)} \left|f(g_h(\cdot,y),y) - f(g_0(\cdot,y),y) \right| \\
& +\left| \int_K^{(h)} f(g_0(\cdot,y),y) - \int_K f(g_0(\cdot,y),y) \right| \\
& \leq |f(\cdot,y)|_{W^{1,\infty}(V(CH))}\|g_h(\cdot,y) - g_0(\cdot,y)\|_{L^{1}(K)} + C|f(g_0(\cdot,y))|_{W^{1,\infty}(K)}|K|h \nonumber \\
& \leq C|f(\cdot,y)|_{W^{1,\infty}(V(CH))}h|K| \\
& \stackrel{\eqref{e:ball RS}}{\leq} C\|f(\cdot,y)\|_{L^{1}(V(\epsilon H))}H^{-d-1}|K|h \\
& \leq C\|f(g_0(\cdot,y),y)\|_{L^{1}(K)}\left({h \over H}\right).
\end{align}
Thus, if $C\left({h \over H}\right)<0.5$, we have that
\begin{align}
\|f(g_h(\cdot,y),y)\|_{L^\infty_h(K)} 
& \leq \|f(\cdot,y)\|_{L^\infty(V(CH))} \\
& \stackrel{\eqref{e:ball RS}}{\leq}
CH^{-d} \|f(\cdot,y)\|_{L^1(V(\epsilon H))} \\
& \leq CH^{-d} \|f(g_0(K),y)\|_{L^1(K)} \\
& \leq 2CH^{-d} \|f(g_h(K),y)\|_{L^1_h(K)}.
\end{align}
\end{proof}

\begin{theorem}[``A priori estimate'' for the uniform discrete reverse Hölder inequality]
Let $U \in \mathbb{C}^d$ be a domain and $L \subset U \cap \mathbb{R}^d$ be compact. Assume that $\partial_x u_{h}^*(t,x) \in L$ for all $0 \leq h \leq h^{(1)}$, $t_0 \leq t \leq \infty$ and $x \in \cl \Omega$. Further assume that the singular values of the Hessian $\partial_x^2 u_{0}^*(t,x)$ are uniformly bounded above and below. Assume that $F = -\log \Phi$ and that $\Phi$ is analytic on $U$. Assume that $\Phi_s$ is uniformly bounded below on $L$. There is a constant $C<\infty$ such that, for all $H \geq h \geq 0$ and $K \in T_H$, and all polynomial functions $v$ such that $Dv$ has degree $\alpha - 1$, then
\begin{align}
\|\sqrt{F''(Dz_{h}^*(t))[(Dv)^2]}\|_{L^{\infty}(K)}
\leq
CH^{-d}\|\sqrt{F''(Dz_{h}^*(t))[(Dv)^2]}\|_{L^{1}(K)}.
\end{align}
This is the uniform discrete reverse Hölder inequality of Definition \ref{d:regular}.
\end{theorem}
\begin{proof}
From $t\Phi - \Phi_s(q,s^{(t)}(q)) = 0$ and the implicit function theorem, we see that $s^{(t)}(q)$ is an analytic function of $(t,q)$. Denoting $w = (q,s^{(t)}(q))$, the function $f^2(q,t,v) = (Dv)^T \Phi^2(w)F''(w)Dv = (\Phi'(w)[Dv])^2 - \Phi(w) \Phi''(w)[(Dv)^2]$ is complex analytic on $U \times Y$ where $Y = [t_0,\infty] \times S$. 
Furthermore,
\begin{align}
\Phi(w) & = {1 \over t}\Phi_s(w) = \Theta(t^{-1}); \label{e:Phi approximation}
\end{align}
i.e. $\Phi(w)t$ is uniformly bounded below and above.
We put $q(x) = g_h(x,t) := \partial_x u_{h}^*(t,x)$. Note that
\begin{align}
\|Du^*_h(t) - Du^*_0(t)\|_{L^{\infty}(\Omega)} \leq Ch.
\end{align}
Now let $\epsilon_0>0$ be as in Lemma \ref{l:epsilon 0}.
If $\epsilon_0 H \leq h = O(h)$ then we use the ``rough'' quadrature bound
\begin{align}
\int_{K}^{(h)} \eta = \sum_i \omega_i \eta(x_i) \geq \omega_{i_0} \eta(x_{i_0}) \geq Ch^d\|\eta\|_{L^{\infty}_h(K)},
\end{align}
for some suitable $i_0$ such that $\|\eta\|_{L^{\infty}_h(K)} = \eta(x_{i_0})$. In this regime, we have $H = O(h)$ so that 
\begin{align}
{
\|\sqrt{F''(Dz_{h}^*(t))[(Dv)^2]}\|_{L^{\infty}_h(K)}
\over
\|\sqrt{F''(Dz_{h}^*(t))[(Dv)^2]}\|_{L^{1}_h(K)}
}
& \leq CH^{-d}.
\end{align}
In the regime $h < \epsilon_0 H$, \eqref{e:hopefully near the end} is the desired estimate.
\end{proof}

\section{Implementation: the practical MGB algorithm} \label{s:practical MGB}

Theorem \ref{t:Algorithm MGB} states that Algorithm MGB converges for certain large $t$ steps, but not arbitrarily large $t$ steps. Thus, some sort of $t$ step size adaptation is needed. Furthermore, it was shown in \citet{loisel2020efficient} that the great majority of the time, $z_h^*(t_{k+1})$ can be computed directly from $z_h^*(t_k)$ with a long step size, with very few Newton steps. In view of these two facts, we now introduce the practical MGB algorithm, which operates as follows.

\begin{definition}[The practical MGB algorithm]
To compute $z_h^*(t_{k})$ from $z_h^*(t_{k-1})$, proceed as follows.
\begin{enumerate}
\item Set $t_k = t_{k-1} \rho_{k-1}$ and attempt to find $z_h^*(t_{k})$ by Newton iteration starting from $z_h^*(t_{k-1})$, with a maximum of 5 Newton iterations allowed. We call this a direct step. Denote by $m_{k,0}$ the number of Newton iterations used here.
\item If the direct step failed to converge in 5 iterations, compute instead $z_h^*(t_{k})$ by the usual MGB algorithm of definition \ref{d:Algorithm MGB}. Denote by $m_{k,{\ell}}$ the number of Newton iterations used on grid level $\ell$.
\item Stepsize adaptation. Denote $m_k = \max_{\ell} m_{k,\ell}$. Set the step size
\begin{align}
\rho_k & = \begin{cases}
\rho_{k-1}^2 & \text{ if } m_k \leq 2, \\
\rho_{k-1} & \text{ if } 3 \leq m_k \leq 5, \\
\sqrt{\rho_{k-1}} & \text{ if } m_k \geq 6.
\end{cases} \label{e:stepsize adaptation}
\end{align}
\end{enumerate}
\end{definition}

Each step of this algorithm requires the minimization of a function by Newton iteration, which we have named $\Bmin(F,c,x^{(0)},R)$. Here, $F$ is the barrier, $c = c[x]$ is a functional of $x$, and $R$ is a matrix whose columns form a basis for the relevant finite element space $V_{h^{(\ell)}} \subset W_0^{1,\infty}(\Omega) \times L^{\infty}(\Omega)$. The function $\Bmin$ uses damped Newton iterations to solve
\begin{align}
\Bmin(F,c,x^{(0)},R) \approx \argmin_{y \in x^{(0)} + \spn R} \int_{\Omega}^{(h)} c[y] + F(Dy).
\end{align}

This architecture allows one to also solve boundary value problems. Indeed, if $x^{(1)} = \Bmin(F,c,x^{(0)},R)$ and $x^{(0)} = (u^{(0)},s^{(0)})$ with $u^{(0)}|_{\partial \Omega} = g \neq 0$, where $g$ is some Dirichlet data, then since $\spn R \subset W^{1,\infty}_0(\Omega)$, we will have that also have $u^{(1)}|_{\partial \Omega} = g \neq 0$. Thus, Dirichlet data that is injected into the first iterate, will be preserved across all subsequent iterates, allowing one to solve inhomogeneous Dirichlet problems.

\subsection{Inhomogeneous Dirichlet problems.}

When solving inhomogeneous Dirichlet problems, it is important to find a smooth prolongation of $g$ to the interior of $\Omega$. To that end, we define $u_h(g)$ to be the solution of the discrete Poisson problem
\begin{align}
\Delta_h u_h = 0 \text{ in } \Omega \text{ and } u_h = g \text{ on } \partial \Omega.
\end{align}
Here, $\Delta_h$ is the usual finite element discretization of the Laplacian on the piecewise polynomial space of degree $\alpha$ on $T_h$.

The user provides boundary data $g$ and a forcing $f$. We then automatically produce an initial value for $z^{(0)} = (u^{(0)},s^{(0)})$ by putting $u^{(0)} = u_h(g)$. For the slack $s^{(0)}$, we initialize it to the constant $s^{(0)} = 1$ and iteratively double it until $F(u^{(0)},s^{(0)}) < \infty$ for all $x \in \Omega$. Given this value of $z^{(0)}$, we may begin the MGB Algorithm to follow the central path.

Although Theorem \ref{t:Algorithm MGB} states that it suffices to choose $t_0 = O(1)$, we found that it is better to use $t_0 \sim h^d$. This seems to result in a more moderate number of initial centering steps needed to locate the central path.

\subsection{Issues of floating point arithmetic.}

One can quickly reach the limits of double precision floating point accuracy. Denote by $\epsilon \approx 2.22\times 10^{-16}$ the ``machine epsilon''. We will be using piecewise quadratic elements in dimension $d=2$, so that our stopping criterion will be $t \sim h^{-4}$. In view of Lemma \ref{c:properties}, the condition number $\kappa$ of $F''(Dz_h^*(t))$ may be as large as $h^{-8}$ and it becomes practically impossible to compute Newton steps if $h^{-8} \sim \epsilon$, i.e. $h \sim 0.01$. At this point, the matrix $H = \cF_h''(Dz_h^*(t))$ becomes numerically singular.

To avoid complications due to floating point roundoff, we regularize our problem as follows. First, we add $10^{-15}\vvvert H\vvvert_{\infty}I$ to $H$, which has a negligible effect on $H$ when it is well-conditioned, but prevents numerical catastrophe when $H$ becomes extremely ill-conditioned. Second, we limit $t$ to $t \leq 10^8$, beyond which it is numerically futile to continue the optimization.

\subsection{The naïve algorithm.}

We have also implemented the naïve algorithm. As it was shown in \citet{loisel2020efficient} that automatic stepsize adaptation is significantly better in practice than the theoretically optimal short step size, we also use the stepsize adaptation \eqref{e:stepsize adaptation} for the naïve algorithm. However, \eqref{e:stepsize adaptation} can only compensate for ``stiffness'' caused by large $t$ steps, and cannot compensate for the difficulty of refining the $h$ parameter when $t$ is already large, as we will see in the numerical experiments.

\section{Numerical experiments}

\begin{figure}
\includegraphics[width=0.5\textwidth]{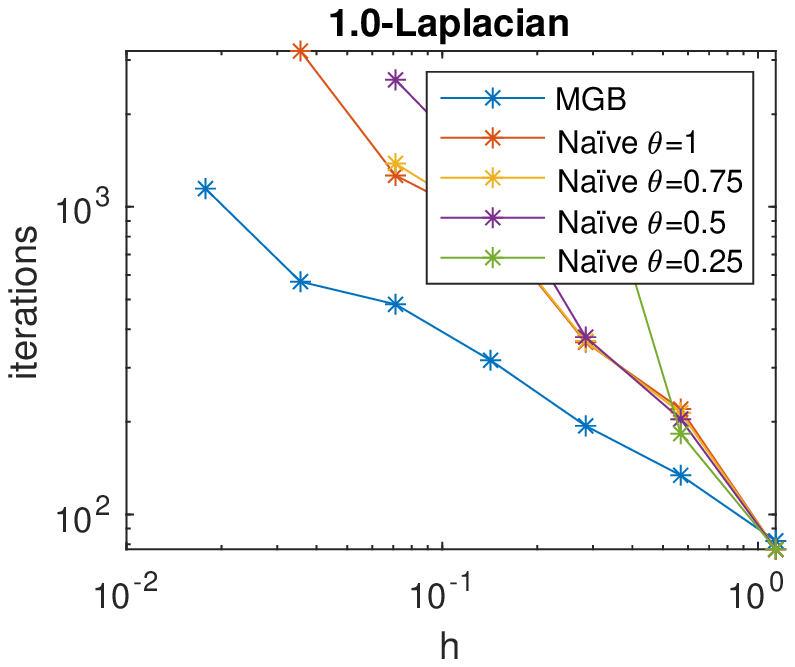}%
\includegraphics[width=0.5\textwidth]{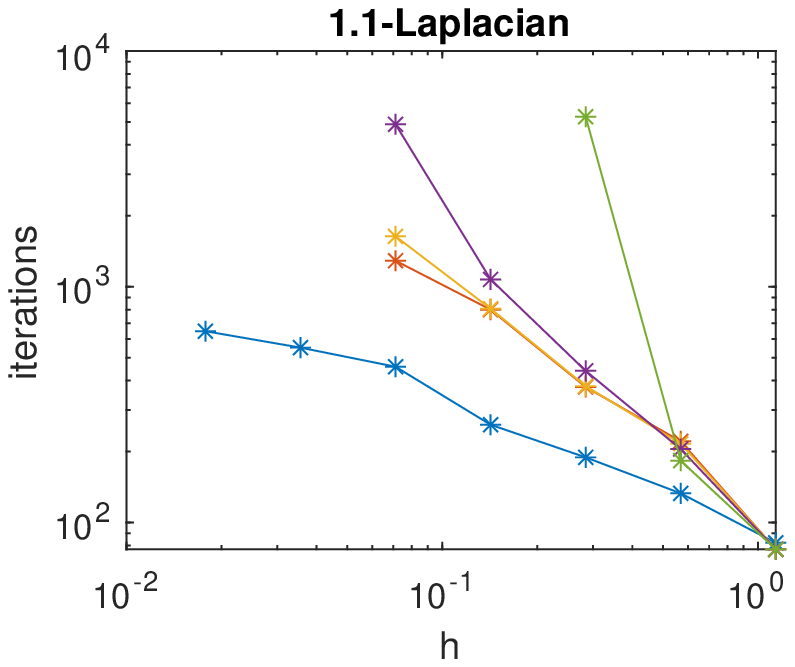}\\ \\
\includegraphics[width=0.5\textwidth]{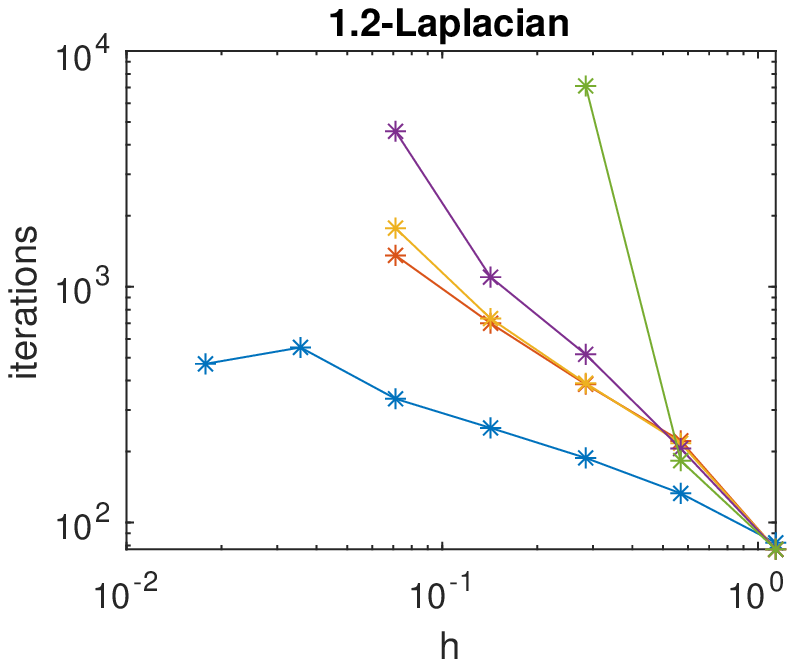}%
\includegraphics[width=0.5\textwidth]{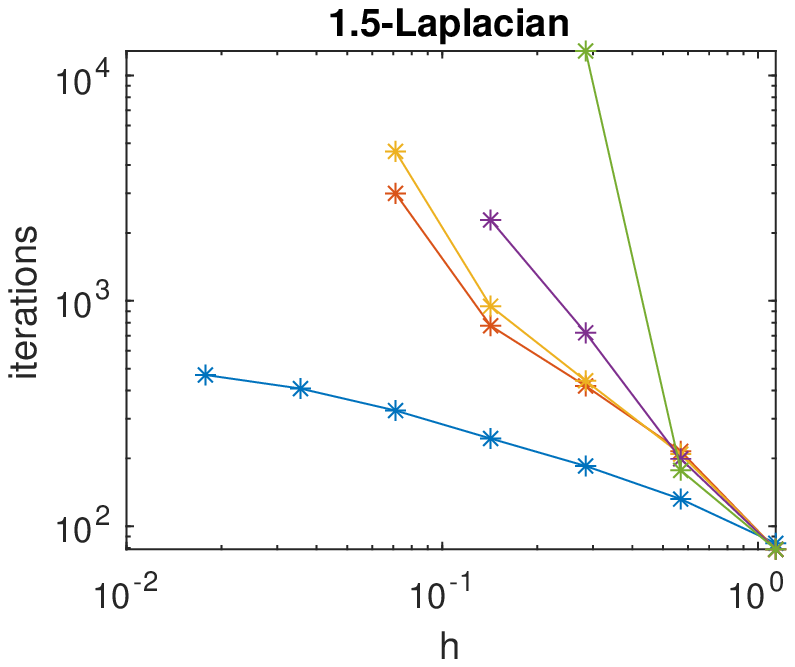}\\ \\
\includegraphics[width=0.5\textwidth]{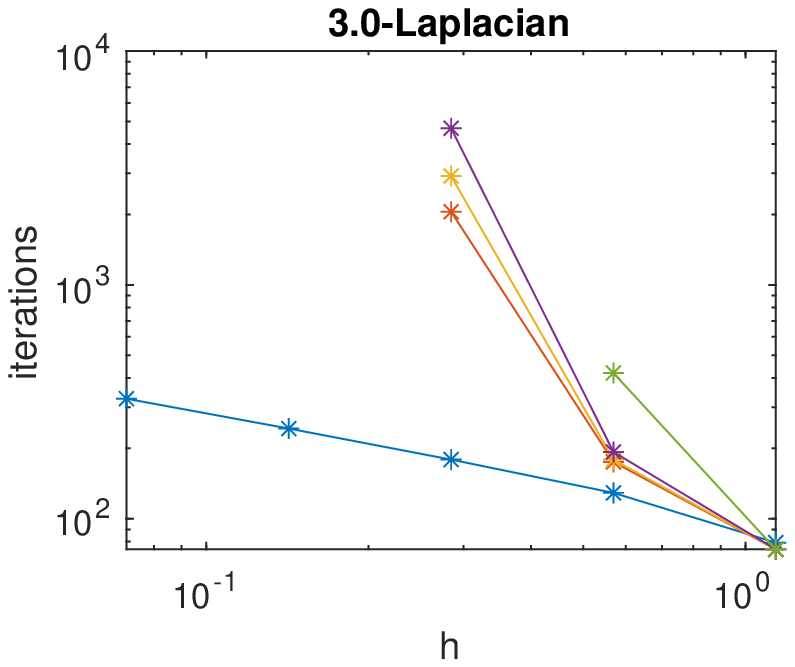}%
\includegraphics[width=0.5\textwidth]{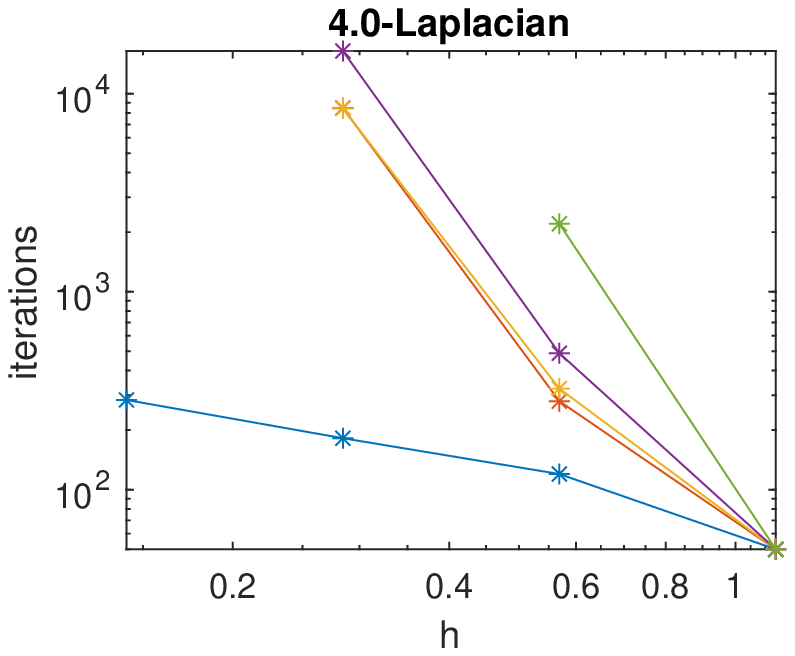}
\caption{Iteration counts of the MGB algorithm, compared to the naïve algorithm with various refinement schedules. \label{f:MGB vs naive}}
\end{figure}
\begin{figure}
\includegraphics[width=0.5\textwidth]{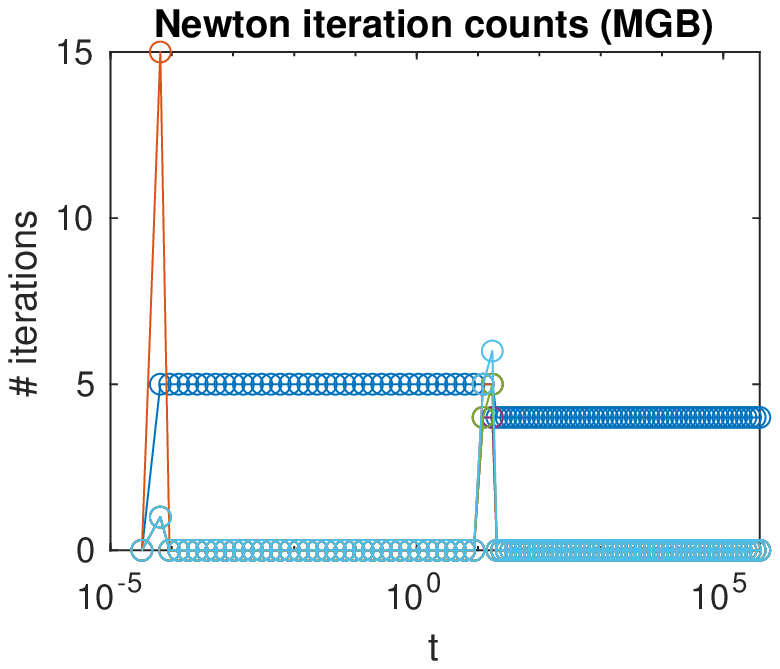}%
\includegraphics[width=0.5\textwidth]{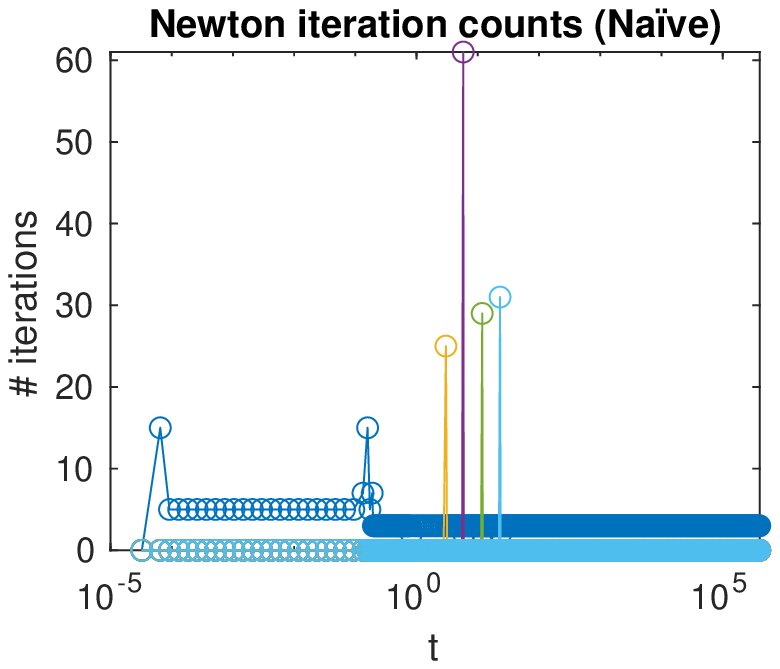}
\caption{Iteration counts for Algorithm MGB (left) and the naïve algorithm (right), as a function of $t_k$, for the $1.0$-Laplacian.}
\end{figure}
\begin{figure}
\includegraphics[width=0.5\textwidth]{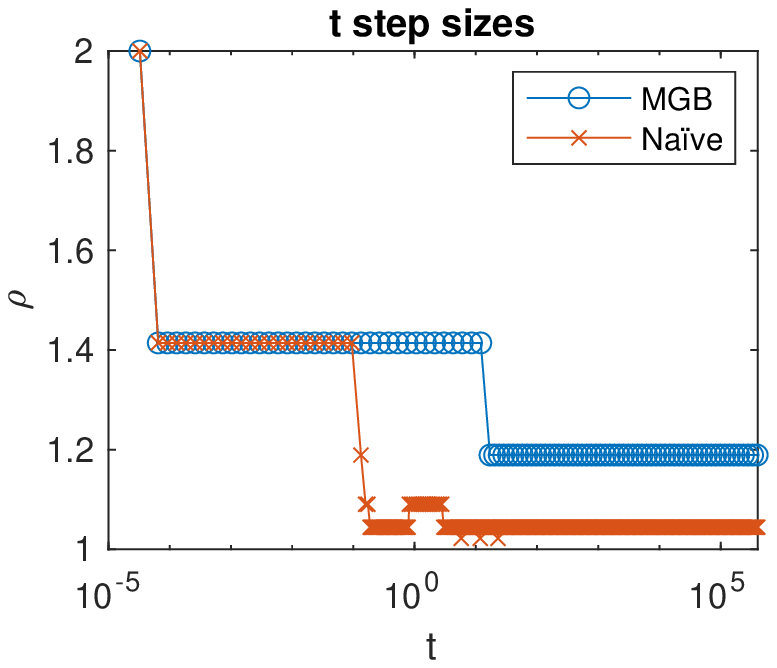}
\caption{Step sizes as a function of $t_k$, for the $1.0$-Laplacian.\label{f:stepsize}}
\end{figure}

We have implemented a suite of tests based on the $p$-Laplacian, parametrized by $1 \leq p < \infty$, with
\begin{align}
\Lambda(q) & = \|q\|_2^p.
\end{align}
We use the self-concordant barrier
\begin{align}
F(q,s) & = -\log(s^{2 \over p} - \|q\|_2^2) - 2\log s,
\end{align}
see \citet{loisel2020efficient} for more information on the $p$-Laplacian and this barrier.

When using homogeneous Dirichlet conditions, a smooth solution $u$ will have some extrema that are interior to $\Omega$, and at those points, one will have $\nabla u(x) = 0$. At these points, if $1 \leq p < 2$, the function $\Lambda(\nabla u(x))$ becomes a distribution, so the forcing must be singular. Since our algorithm does not handle distributional data, we prefer to solve a problem with boundary value $g$ on $\partial \Omega$, and forcing $f=0$. Specifically

We report the iteration counts in Figure \ref{f:MGB vs naive}. We vary $p \in \{1,1.1,1.2,1.5,3,4\}$, $0.01 < h \leq 1.3$, and compare Algorithm MGB and the naïve algorithm with a range of $h$ and $t$ refinement schedules. We report the number of iterations needed in each case to obtain convergence. Each algorithm is stopped if it runs longer than 5 minutes, in which case it is deemed to have failed to converge.

The naïve algorithm is parametrized by its $h$ and $t$ refinement schedule. We have used the schedule $\ell = \theta \log_2 t$, where $\ell$ denotes the grid level $h^{(\ell)}$. The algorithm alluded to by \cite{schiela2011interior} can be related to the case $\theta = 0.25$; indeed, in that scenario, if large $t$ steps are used throughout and if the $h$ refinements converge quickly, then indeed most of the iterations will occur on the coarse grid levels. When $\theta \geq 0.5$, at least half of the iterations are expected to be computed on the fine grid.

Unfortunately, all the versions of the naïve algorithm have trouble converging for small values of $h$. The failures are all caused by an extremely large number of Newton iterations required to perform the $h$ refinement when $t$ becomes large. Note that when failures occur because of large $t$ step sizes, then smaller $t$ step sizes can be used to allow the algorithm to converge, but with $h$ refinements, it is impossible to find intermediate grid levels between a level $h^{(\ell)}$ and the next one $h^{(\ell+1)} = {1 \over 2}h^{(\ell)}$. As a result, the only way of making the naïve algorithm work is to refine the $h$ grid earlier, e.g. as per the $h$-then-$t$ schedule.

The MGB algorithm converges in all cases $p < 2$ and for all grid parameters, and the convergence is quite fast.  Algorithm MGB does not converge for all values of $h$ when $p>2$, but this is expected because of floating point loss of accuracy in these scenarios, as noted in \cite{loisel2020efficient}. Briefly speaking, it is difficult to precisely locate the minimum of the function $|x|^p$ when $p$ is large. Despite this, Algorithm MGB is able to converge faster in a wider array of situations, than the naïve algorithms.

We have also displayed how many Newton iterations are needed at each $t_k$, for Algorithm MGB and the naïve algorithm for the $1$-Laplacian. We have displayed iteration counts in different color for each grid level $\ell$. We see that the naïve algorithm's iteration count skyrockets up to 60 when grid transitions occur. By contrast, Algorithm MGB never needs more than $15$ iterations, and then only for the very first iteration, which is expected to take $O(1)$ iterations. We also notice that, at $t = 17.0667$, algorithm MGB used 6 iterations on the coarsest grid level. This is the only iteration, for this specific problem instance, where the ``direct step'' described in Section \ref{s:practical MGB} failed to converge, and the full MGB step was used instead.

As can be seen in Figure \ref{f:stepsize}, this MGB step allows the path-following algorithm to take large steps at all values of $t$, and the step size $\rho$ never decreases to less than $1.18$. By contrast, the naïve algorithm struggles and resorts to step sizes $\rho \approx 1.02$. In this case, both algorithm converged, but for larger problems (with smaller values of $h$), the naïve algorithm fails to converge by taking too many Newton iterations and too much time.

\section{Conclusions and outlook}

Algorithm MGB is the first algorithm that is a provably optimal solver (in the big-$\tilde{O}$ sense) for convex Euler-Lagrange problems, or nonlinear elliptic PDEs. Its running time is shown to be $\tilde{O}(n)$ FLOPS. Numerical experiments confirm the analysis.

\bibliography{mybib}
\end{document}